\documentclass[11pt,a4paper]{article}
\usepackage{amsfonts,amsmath,amsthm,amssymb,stmaryrd}
\usepackage{graphicx}\usepackage{bm}
\usepackage{esint}
\usepackage{hyperref}
\usepackage[RGB]{xcolor}
\usepackage[notcite, notref]{showkeys}
\usepackage{esint}
\usepackage{epstopdf}
\usepackage{mathrsfs}
\usepackage{amscd}
\usepackage{setspace}
\usepackage[title]{appendix}

\newtheorem{Lemma}{Lemma}[section]
\newtheorem{Theorem}{Theorem}[section]
\newtheorem{Definition}{Definition}[section]
\newtheorem{Proposition}{Proposition}[section]

\newtheorem{Corollary}{Corollary}[section]

\numberwithin{equation}{section} \allowdisplaybreaks

\usepackage{indentfirst}
\usepackage[T1]{fontenc}
\usepackage[utf8]{inputenc}
\usepackage{authblk}

\newcommand{\LC}{\left(}
\newcommand{\RC}{\right)}

\def\R{\mathop{\mathbb R\kern 0pt}\nolimits}
\def\p{\partial}

\newcommand{\beno}{\begin{eqnarray*}}
\newcommand{\eeno}{\end{eqnarray*}}

\let\al=\alpha
\let\ga=\gamma

\def\be{\begin{equation}}
\def\beq{\begin{equation}}
\def\bel{\begin{equation}\label}
\def\eeq{\end{equation}}
\def\bega{\begin{array}}
\def\enda{\end{array}}
\def\begi{\begin{itemize}}
\def\endi{\end{itemize}}

\begin{document}
\date{}

\title{\bf Lipschitz optimal transport metric for a wave system modeling nematic liquid crystals}

\author[1]{Hong Cai\thanks{Email: caihong19890418@163.com (Hong Cai)}}
\author[2]{Geng Chen\thanks{Email: gengchen@ku.edu (Geng Chen)}}
\author[2]{Yannan Shen\thanks{Email: yshen@ku.edu (Yannan Shen)}}

\affil[1]{\it{\small School of Mathematics and Physics, Qingdao University of Science and Technology, Qingdao, 266061, P.R. China.}}
\affil[2]{\it{\small Department of Mathematics, University of Kansas, Lawrence, KS 66045, USA. }}

\renewcommand\Authands{ and }



\maketitle

\begin{abstract}
In this paper, we study the Lipschitz continuous dependence of conservative H\"older continuous weak solutions to a variational wave system derived from a model for nematic liquid crystals. Since the solution of this system generally forms finite time cusp singularity, the solution flow is not Lipschitz continuous under the Sobolev metric used in the existence and uniqueness theory. We establish a Finsler type optimal transport metric, and show the Lipschitz continuous dependence of solution on the initial data under this metric. This kind of Finsler type optimal transport metrics was first established in [A. Bressan and G. Chen, Arch. Ration. Mech. Anal. 226(3) (2017), 1303-1343]  for the scalar variational wave equation. This equation can be used to describe the unit direction $\mathbf{n}$ of  mean orientation of nematic liquid crystals, when $\mathbf{n}$ is restricted on a circle. The model considered in this paper describes the propagation of $\mathbf{n}$ without this restriction, i.e. $\mathbf{n}$ takes any value on the unite sphere. So we need to consider a wave system instead of a scalar equation.
 \bigbreak
\noindent

\bigbreak

\end{abstract}

\section{Introduction}
\setcounter{equation}{0}
In this paper, we study the Lipschitz continuous dependence for solutions of the variational wave system
\begin{equation} \label{vwl}
\partial_{tt}n_i-\partial_x\big(c^2(n_1)\partial_x n_i\big)=\bigl(-|{\mathbf n}_t|^2+\big(2c^2(n_1)-\zeta_i\big)|{\mathbf n}_x|^2\bigr)n_i,\qquad
i=1,2,3.
\end{equation}
The time $t$ and space variable $x$ belong to $\R^+$ and $\R$, respectively,
and the unit vector ${\mathbf n}=(n_1,n_2, n_3)$ satisfies
\begin{equation}\label{n1}
|{\mathbf n}|=1.
\end{equation}
The (positive) wave speed $c$ depends on $n_1$ with
\beq\label{c-def}
c^2(n_1)=\alpha+(\gamma-\alpha)n_1^2.
\eeq
The constants
\[
 \zeta_1=\ga>0\quad\hbox{and}\quad\zeta_2=\zeta_3=\al>0.
 \]
In this paper, we consider the initial value problem with inital data satisfying
\begin{equation}\label{ID}
 n_i|_{t=0}={n_i}_0\in H^1(\mathbb{R}),\quad
(n_i)_t|_{t=0}={n_i}_1\in L^2(\mathbb{R}), \quad i=1, 2, 3.
\end{equation}

We briefly introduce the origin of system \eqref{vwl} from modelling nematic liquid crystal. Liquid crystal is often viewed as an intermediate state between liquid and solid. More precisely, a nematic crystal can be described, when we ignore the motion of the fluid, by the dynamics of the so-called director field of unit vectors ${\mathbf n}\in{\mathbb S}^2$ describing the orientation of the rod-like molecules. We consider a
regime in which inertia effects dominate viscosity. The
propagation of the orientation waves in the director field is
modelled by the least action principle (\cite{AH,[37]})
\beq
\delta \int_{\R^+} \int_{\R^3}\Big\{ \frac12 \p_t{\mathbf
n}\cdot\p_t{\mathbf n} - W({\mathbf n}, \nabla{\mathbf
n})\Big\}\,d{\mathbf x}\,dt  = 0, \qquad {\mathbf n}\cdot
{\mathbf n} = 1. \label{1.2}
\eeq
The potential energy density $W$ is given by the well-known Oseen-Frank
energy from the continuum theory of nematic liquid
crystals (\cite{GP}, Ch. 3.),
\beq
W\left({\mathbf n},\nabla{\mathbf n}\right) =
\frac12\alpha(\nabla\cdot{\mathbf n})^2 +\frac12\beta\left({\mathbf n}\cdot(\nabla\times{\mathbf
n})\right)^2
+\frac12\gamma\left|{\mathbf n}\times(\nabla\times{\mathbf
n})\right|^2, \label{1.1a}
\eeq
where the positive constants $\alpha$, $\beta$, and $\gamma$ are elastic constants
of the liquid crystal, corresponding to splay, twist, and bend, respectively. A special case is the one-constant model in which $\alpha=\beta=\gamma$, the function $W$ then reduces to the harmonic map energy density $W=\frac{1}{2}\alpha|\nabla \mathbf{n}|^2.$
The associate variational principle \eqref{1.2} leads to the equation for harmonic wave maps from ($1+3$)-dimensional Minkowski space into two sphere, see \cite{CTZ,S,STZ} for example.

The Euler-Lagrange equation associated with \eqref{1.2} and \eqref{1.1a} is
\begin{equation}\label{EL}
{\mathbf n}_{tt}=\alpha \nabla(\nabla\cdot \mathbf n)-\beta [A \nabla\times \mathbf n+\nabla\times(A \mathbf n)]+\gamma [B\times (\nabla\times \mathbf n)-\nabla\times (B\times \mathbf n)]+\lambda \mathbf n,
\end{equation}
with $A=\mathbf{n}\cdot (\nabla\times \mathbf n),\  B=\mathbf n\times (\nabla\times \mathbf n).$
The Lagrange multiplier $\lambda(x, t)$ in \eqref{EL} is chosen so that $\mathbf n\cdot\mathbf n = 1$, and is given explicitly in terms of $\mathbf n$ by
\begin{equation}\label{Lm}
\lambda=-|{\mathbf n}_{t}|^2+\alpha[|\nabla\mathbf n|^2-|\nabla\times \mathbf n|^2]+2[\beta A^2 +\gamma|B|^2]+(\alpha-\gamma)(\nabla\cdot B).
\end{equation}

When the space dimension is one (1-d), i.e. $x\in \R$, $W$ in \eqref{1.1a} is given specifically by
\begin{equation*}
W\left({\mathbf n},\partial_x{\mathbf n}\right) =
\frac\alpha 2(\partial_x n_1)^2 +\frac\beta 2\left((\partial_x n_2)^2+(\partial_x n_3)^2\right)
+\frac12(\gamma-\beta)n_1^2 |\mathbf{n}_x|^2,
\end{equation*}
which together with \eqref{EL} and \eqref{Lm} implies that
\begin{equation}\label{full}
\begin{cases}
\partial_{tt}n_1-\partial_{x}[c_1^2(n_1)\partial_x n_1]=\left[-|\mathbf{n}_t|^2+(2c_2^2-\gamma)|\mathbf{n}_x|^2+2(\alpha-\beta)
(\partial_x n_1)^2\right]n_1,\\
\partial_{tt}n_2-\partial_{x}[c_2^2(n_1)\partial_x n_2]=\left[-|\mathbf{n}_t|^2+(2c_2^2-\beta)|\mathbf{n}_x|^2+(\beta-\alpha)
n_1\partial_{xx}n_1\right]n_2,\\
\partial_{tt}n_3-\partial_{x}[c_2^2(n_1)\partial_x n_3]=\left[-|\mathbf{n}_t|^2+(2c_2^2-\beta)|\mathbf{n}_x|^2+(\beta-\alpha)
n_1\partial_{xx}n_1\right]n_3,
\end{cases}
\end{equation}
with $c_1^2(n_1)=\alpha+(\gamma-\alpha)n_1^2$ and $
c_2^2(n_1)=\beta+(\gamma-\beta)n_1^2.$
In particular, putting $\alpha=\beta$ in \eqref{full}, we obtain our system \eqref{vwl}.

The existence and uniqueness of energy conservation $H^1$ solution for \eqref{vwl}--\eqref{ID} has already been established in \cite{CZZ} by Chen-Zhang-Zheng following by an earlier work \cite{ZZ10}, and in \cite{CCD} by Cai-Chen-Du, respectively. In general, the solution of  \eqref{vwl}--\eqref{ID} or other type of variational wave system such as \eqref{VW} is not unique, due to the formation of cusp singularity \cite{BH,BZ,GHZ,ZZ03}. To obtain a unique solution after the formation of singularity, one needs to assume an additional admissible condition, such as the energy conservative condition in the week form. Also see results for dissipative solutions in \cite{BH, ZZ03}, and for \eqref{full} in \cite{ZZ11}.

Since the solution of \eqref{vwl}--\eqref{ID} generally forms finite time cusp singularity, the solution flow is not Lipschitz continuous under the $H^1$ metric, \cite{GHZ}.
The goal of this paper is to establish a Finsler type optimal transport metric, and show the Lipschitz continuous dependence of conservative solution on the initial perturbation under this metric.

In what follows, we always assume that the following generic condition is satisfied
\begin{equation}\label{gencon}
\alpha\neq \gamma.
\end{equation}
When $\alpha=\gamma$, the wave speed $c$ is a constant, so the system becomes a one dimensional semi-linear wave equation. Then the well-posedness of solution can be solved easily using the classical method. To avoid unnecessary complexity on notations and estimates, we do not address this case in this paper.

There is a highly simplified case when $
{\mathbf n} =  (\cos u(t,x), \sin u(t,x),0)
$ (planar deformation) with $x\in \mathbb R$, where the dependent variable $u\in\mathbb{R}$ measures the angle of the director field to the $x$-direction.
In this case, the function $u$ satisfies the scalar variational wave equation
\beq\label{VW}
u_{tt} -c(u)(c(u)\,u_x)_x = 0,\eeq
with
$
c^2(u) = \gamma\cos^2u + \alpha\sin^2u.
$
See \cite{AH, BZ, CZZ} for more details on the derivations of \eqref{vwl} and \eqref{VW}. The research on global well-posedness of H\"older continuous conservative solutions for variational wave type equations was initiated from \eqref{VW}, where current results include global existence \cite{BZ,HR}, uniqueness \cite{BCZ}, Lipschitz continuous dependence under Finsler type transport metric \cite{BC2015}, and generic regularity \cite{BC}.

Especially, the construction of new Lipschitz optimal transport metric for \eqref{vwl} is based on the metric established  for \eqref{VW} in \cite{BC2015} by Bressan and the second author.
In this paper, we leap from a scalar equation to a system of wave equations. The new metric for system \eqref{vwl} is quite different from the one for scalar equation, mainly because we need to control the energy transfer between different components of ${\bf n}$ in each characteristic family. We will introduce more details in section 2.

Finally, the recent result on Poiseuille flow of nematic liquid crystals via the full Ericksen-Leslie model in \cite{CHL20} shows that the results on global well-posedness for variational wave systems \eqref{VW} and \eqref{vwl} have direct applications on the Ericksen-Leslie model described by a coupled system consisting of a wave system on the director field of unit vector ${\bf n}$ and Navier-Stokes equations on  the fluid velocity ${\bf u}$. For results on elliptic and parabolic type Ericksen-Leslie systems, which are proved by very different techniques, we refer the reader to the pioneer paper \cite{lin89}, a survey paper \cite{linwangs14} and the references therein.

\subsection{Existing existence and uniqueness results}

In \cite{CCD, CZZ}, the authors established the existence and uniqueness of global conservative solution to the Cauchy problem \eqref{vwl}--\eqref{ID}. We first review the global existence theorem in \cite{CZZ}, where one can also find the definition of weak solution inside this theorem.

\begin{Theorem} [Existence \cite{CZZ}] \label{CZZthm}
The Cauchy problem \eqref{vwl}--\eqref{ID} has a global weak solution ${\mathbf n}(t, x)=(n_1, n_2, n_3)(t,x)$ defined for all $(t, x)\in [0, \infty)\times {\mathbb R}$ in the following sense:
\begin{itemize}
\item[{\rm (i)}] In the $t$-$x$ plane, the functions $(n_1, n_2, n_3)$ are locally H\"older continuous
with exponent $1/2$.   This solution $t\mapsto (n_1, n_2, n_3)(t,\cdot)$ is
continuously differentiable as a map with values in $L^p_{\rm
loc}$, for all $1\leq p<2$. Moreover, it is Lipschitz continuous with respect to
$(w.r.t.)$~the $L^2$ distance, that is, there exists a constant $L$ such that
\begin{equation*}
 \big\|{n_i}(t,\cdot)-{n_i}(s,\cdot)\big\|_{L^2}
 \leq L\,|t-s|, \quad i=1, 2, 3,\quad \hbox{for all } t,s\in\mathbb R^+.
\end{equation*}

\item[{\rm (ii)}] The functions $(n_1, n_2, n_3)(t,x)$ take on the initial conditions in \eqref{ID}
pointwise, while their temporal derivatives hold in  $L^p_{\rm
loc}\,$ for $p\in [1,2)\,$.
\item[{\rm (iii)}] The equation \eqref{vwl} holds in distributional sense for all test function $\varphi\in
C^1_c(\mathbb R^+\times \mathbb R)$.
\end{itemize}
\end{Theorem}

The uniqueness result for conservative solution in \cite{CCD} can be summarized as follows.

\begin{Theorem}[Uniqueness \cite{CCD} and energy conservation \cite{CZZ}]\label{ECthm}
 Under the previous assumptions, a unique solution ${\mathbf n}={\mathbf n}(t,x)$ exists which is {\bf conservative} in the following sense:

 There exist two  families of positive Radon measures
 on the real line: $\{\mu_-^t\}$ and $\{\mu_+^t\}$, depending continuously on $t$ in the weak topology of measures, with the following properties.
\begi
\item[(i)] At every time $t$ one has
\begin{equation*}
\mu_-^t(\mathbb {R})+\mu_+^t(\mathbb {R})~=~E_0~:=~2
\int_{-\infty}^\infty \Big[|{\bf n}_1|^2(x) + c^2(n_{10}(x)) |{\bf n}_{0,x}(x)|^2\Big]\, dx \,,\end{equation*}
where we denote the initial data
$$
{\bf n}_0=(n_{10},n_{20}, n_{30})= {\bf n}|_{t=0},\qquad
{\bf n}_1=(n_{11},n_{21}, n_{31})= {\bf n}_t|_{t=0}.
$$

\item[(ii)] For each $t$, the absolutely continuous parts of $\mu_-^t$ and
$\mu_+^t$ w.r.t.~the Lebesgue measure
have densities  respectively given  by
$
\bigl|{\bf n}_t + c(n_1) {\bf n}_x\bigr|^2 {\rm ~and~}
\bigl|{\bf n}_t - c(n_1) {\bf n}_x\bigr|^2.
$
\item[(iii)]  For almost every $t\in\mathbb {R}^+$, the singular parts of $\mu^t_-$ and $\mu^t_+$
are concentrated on the set where $n_1=0$ or $\pm 1$, when $\alpha\neq\gamma$.
\endi
\end{Theorem}

\subsection{Our main result}
Then we come to state our main Lipschitz continuous dependence theorem.
\begin{Theorem}\label{thm_metric}
The energy conservative weak solution to the nonlinear wave system of nematic liquid crystals \eqref{vwl}--\eqref{ID} depends Lipschitz continuously on the initial data, under a Finsler type optimal transport metric, defined in Definition \ref{def_weak}.
 Namely, let $(\mathbf{n}_0,\mathbf{n}_1)$ and $(\hat{\mathbf{n}}_0,\hat{\mathbf{n}}_1)$ be two initial data in \eqref{ID}, then for any time $t\in[0,T]$, there exists a distance functional $d$, such that, the corresponding solutions satisfy
\[d\big((\mathbf{n},\mathbf{n}_t)(t),(\hat{\mathbf{n}}, \hat{\mathbf{n}}_t)(t)\big)\leq C\cdot d\big((\mathbf{n}_0,\mathbf{n}_1),(\hat{\mathbf{n}}_0, \hat{\mathbf{n}}_1)\big),\]
where the constant $C>0$ depends only on $T$ and initial total energy.
\end{Theorem}


This paper is divided into five sections. In section 2, we introduce the main ideas used to construct the metric, and the difference between our metric and the metric for scalar equation. In section 3, we establish the Lipschitz metric for smooth solutions.
In section 4, we extend the metric to piecewise smooth generic solutions, using the generic regularity result in \cite{CCD}. Finally, we extend the metric to $H^1$ solution and prove the main theorem in section 5, where we also compare our Finsler metric with some Sobolev metrics and Kantorovich-Rubinstein metric.

\section{Basic setup and main ideas used to establish the metric}\label{sec_smooth}
To describe how to construct the Lipschitz metric, we first consider smooth solutions to \eqref{vwl}--\eqref{ID}. Due to energy concentration when singularity forms, the solution flow fails to be Lipschitz under the $H^1$ distance, where this distance is a natural choice corresponding to the energy. Instead, we will establish a Finsler type optimal transport geodesic distance between any two solutions. Basically, the optimization is taken on the cost of energy transportation between two solutions.

To keep track of the cost of energy transportation, we are led to construct the geodesic distance. That is, for two given solution profiles $\mathbf{n}(t)$ and $\mathbf{n}^\epsilon(t)$, we consider all possible smooth deformations/paths $\gamma^t: \theta \mapsto \mathbf{n}^\theta(t)$ for $\theta\in [0,1]$ with $\gamma^t(0) = \mathbf{n}(t)$ and $\gamma^t(1) = \mathbf{n}^\epsilon(t)$, and then measure the length of these paths through integrating the norm of the tangent vector $d\gamma^t/d\theta$. The distance between $\mathbf{n}$ and $\mathbf{n}^\epsilon$ will be calculated by the optimal path length
\begin{equation*}
d\LC \mathbf{n}(t), \mathbf{n}^\epsilon(t) \RC = \inf_{\gamma^t}\|\gamma^t\| := \inf_{\gamma^t}\int^1_0 \| \mathbf{v}^\theta(t) \|_{\mathbf{n}^\theta(t)} \ d\theta, \quad \text{where } \mathbf{v}^\theta(t) = {d\gamma^t\over d\theta}.
\end{equation*}
Here the subscript $\mathbf{n}^\theta(t)$ emphasizes the dependence of the norm on the flow $\mathbf{n}^\theta$. In fact, there might be no smooth enough path between two solutions. We will use the generic regularity result in \cite{CCD} to overcome this problem.


%

The metric will be established in three steps:

\begin{itemize}

\item[1.]
For smooth solutions,
we find a norm $\| \mathbf{v}^\theta(t) \|_{\mathbf{n}^\theta(t)}$ measuring the cost in shifting from one conservative solution $(\mathbf{n},\mathbf{n}_t)$  (with energy density $\mu$) to the other one $(\hat{\mathbf{n}},{\hat{\mathbf{n}}}_t)$, for any  time $t\leq T$, such that
\beq\label{2.41}
\frac{d}{dt}\| \mathbf{v}^\theta(t) \|_{\mathbf{n}^\theta(t)}\leq C_T\cdot \|\mathbf{v}^\theta(t) \|_{\mathbf{n}^\theta(t)}.
\eeq
Hence,
\begin{equation*}\label{d2}\textstyle d\big( \mathbf{n}(t),\hat{\mathbf{n}}(t))\big)
~\leq~C_T\cdot
d\big( \mathbf{n}(0), \hat{\mathbf{n}}(0))\big),\end{equation*}
where $C_T$ is a constant only depending on the arbitrarily given $T$ and initial energy, but is uniformly bounded when solution approaches a singularity.
\item[2.] Extend the Lipschitz metric in step 1 to piecewise smooth generic solutions.
\item[3.] Apply the generic regularity result in \cite{CCD} to prove the desired Lipschitz continuous property for any (piecewise smooth) generic solution, then take a limit to all conservative solutions, using the result in \cite{CCD} that generic solutions are dense in the energy space $(\mathbf{n},\mathbf{n}_t)(t)\in H^1\times L^2$.
\end{itemize}




Finally, we introduce how to define $\| \mathbf{v}^\theta(t) \|_{\mathbf{n}^\theta(t)}$ in Step 1 such that the inequality \eqref{2.41} is uniformly satisfied before blowup.  
To embed the wave structure in the metric, we consider  a ``double transportation problem'',  i.e.  study the wave propagation for forward and backward characteristics, respectively, then find two corresponding cost functions.

We introduce
\begin{equation} \left\{
\begin{array}{l}
 \mathbf{R}=(R_1,R_2,R_3) ~:=~{\mathbf n}_t+c{\mathbf n}_x,\\ [4mm]
 \mathbf{S}=(S_1,S_2,S_3)~:=~{\mathbf n}_t-c{\mathbf n}_x,
 \end{array}\right.    \label{R-S}
 \end{equation}
for backward and forward characteristic directions, respectively. And the corresponding energy densities are ${\bf R}^2$ and ${\bf S}^2$, where we use the following notations
\[
 \mathbf{R}^2= \mathbf{R}\cdot \mathbf{R},\qquad  \mathbf{S}^2= \mathbf{S}\cdot \mathbf{S}.
\]
Then, for smooth solutions, equations (\ref{vwl}) is equivalent to the following system for $(\mathbf{R}, \mathbf{S}, \mathbf{n})$
\begin{equation} \left\{
\begin{array}{ll}
\partial_t R_i-c\partial_xR_i=
\dfrac1{4c^2}\bigl\{(c^2-\zeta_i)(\mathbf{R}^2+\mathbf{S}^2)-2(3c^2-\zeta_i)
\mathbf{R}\cdot\mathbf{S}\bigr\}n_i+\dfrac{c'(n_1)}{2c(n_1)}(R_i-S_i)R_1, \\
\partial_t S_i+c\partial_xS_i=\dfrac1{4c^2}\bigl\{(c^2-\zeta_i)
(\mathbf{R}^2+\mathbf{S}^2)-2(3c^2-\zeta_i)\mathbf{R}\cdot\mathbf{S}\bigr\}n_i-\dfrac{c'(n_1)}{2c(n_1)}(R_i-S_i)S_1, \\ [4mm]
{\mathbf n}_x=\dfrac{\mathbf{R}-\mathbf{S}}{2c(n_1)}\quad \mbox{or} \quad {\mathbf n}_t=\dfrac{\mathbf{R}+\mathbf{S}}2,
 \end{array}\right. \label{R-S-eqn}
 \end{equation}
 for $i=1,2,3$, with
 $
\zeta_1=\ga$ and $\zeta_2=\zeta_3=\al.$
System (\ref{R-S-eqn}) has the following form of energy conservation law:
\begin{equation*}\label{energy1}
\frac{1}{4}\partial_t\bigl(\mathbf{R}^2+\mathbf{S}^2\bigr)-\frac{1}{4}\partial_x\bigl(c(n_1)(\mathbf{R}^2-\mathbf{S}^2)\bigr)=0,
\end{equation*}
and two balance laws for energy densities in two directions, respectively,
\beq\label{balance}
\left\{
\begin{array}{rcl}(
\mathbf{R}^2)_t - (c\mathbf{R}^2)_x & = & {c'(n_1)\over 2c(n_1)}( \mathbf{R}^2S_1 - R_1 \mathbf{S}^2)\, , \\ [3mm]
(\mathbf{S}^2)_t + (c\mathbf{S}^2)_x & = & - {c'(n_1)\over 2c(n_1)}( \mathbf{R}^2S_1 -R_1 \mathbf{S}^2)\,.
\end{array}
\right.
\eeq

Now we introduce difficulties we meet and new ideas we use when we establish the metric.
\paragraph{\bf 1.}
The double transportation problem gives us tools, i.e. equations  \eqref{R-S-eqn} and \eqref{balance} on  $\bf R$ and $\bf S$, to study wave propagation in each characteristic family and wave interactions.  This is crucial for us to find cost functions and prove the Lipschitz continuous property.

However, forward and backward energy might increase in the wave interaction (see the cubic nonlinearity in \eqref{balance}), although the total energy is bounded.  This happens because energy transfers between different characteristic families during wave interactions.

We introduce some interaction potentials, which share similar philosophy as the Glimm potential for hyperbolic conservation laws.  Very roughly speaking, the interaction potential memories the possible future increase of energy on a single forward or backward wave.  As a wave interaction happens,  the interaction potential (future possible increase of energy) decays,  since the current interaction is out of the list of future interactions. This decay will balance the possible increase of forward or backward energy.
\paragraph{\bf 2.}
The second difficulty comes from the energy transfer inside one characteristic direction between different components. The quadratic terms $R_i R_1$ and $S_i S_1$ in the equations of $R_i$ and $S_i$ in \eqref{R-S-eqn} shows such kind of phenomena.  This is a fundamental difficulty when one jumps from a scalar wave equation to a wave system.

The wave potentials mentioned in the last part can only balance higher order crossing terms such as $S_i R_1$ in \eqref{R-S-eqn} or ${\bf R}^2 S_1 $ in \eqref{balance}. But it takes no effect on $R_i R_1$, $i\neq 1$, in \eqref{R-S-eqn}.

Briefly speaking, our strategy is to adjust components in the metric in a very subtle way.  This is the most difficulty part in this paper, and will make the metric for \eqref{vwl} quite different from the one for \eqref{VW} in \cite{BC2015}.

The most important discovery in this paper is that we find the cancellation between  time derivatives of two terms in the metric for scalar equation \eqref{VW} ($\dot{I}_2$ and $\dot{I}_5$ in \cite{BC2015}). This cancellation also holds for system \eqref{vwl}. Although for a scalar variational wave equations, one can bound these two time derivatives separately, it is not the case for the wave system \eqref{vwl} because of the energy transfer in the same characteristic family between different components. After using the new term in the metric, now denoted as $I_4$ in \eqref{norm1} (correspond to $I_2+I_5$ in  \cite{BC2015}), one can prove \eqref{2.41}. In fact, we find the new term $I_4$ exactly accounts for the change of base measure with density $R_i$. This is a more appropriate term to use in the metric than the old two terms used in \cite{BC2015}, although each of them also has its physical meaning.
%

Secondly, wave speed $c(n_1)$ only depends on $n_1$, and the equations of $n_i$ have different coefficients $\zeta_i$. As a consequence, our metric needs to be ``inhomogeneous'' in order to reflect the inhomogeneity mentioned above.
Let's only explain the idea for the backward wave on $\bf R$. The idea for the forward direction is the same.
To obtain precise estimates on the propagation of each $R_i$ and energy transfer between $R_i$ and $R_j$ with $i\neq j$, we need to adjust the relative shift term, such as change $R_i$ to $R_1$ in some relative shift term to reflect the dependence of $c(n_1)$ on $n_1$ but not on $n_2$ and $n_3$. In fact, after we shift a wave, we create some wave interactions manually, so the corresponding increase of energy needs to be counted in the metric, by adding some relative shift term. This is a very crucial and subtle part in the metric. More details will be introduced later when we construct the metric.

%
\section{The norm of tangent vectors for smooth solutions}
Now, let us consider a smooth solution $(\mathbf{n},\mathbf{R},\mathbf{S})(x)$ to \eqref{vwl}, \eqref{R-S-eqn}, and then take a family of perturbed solutions
$(\mathbf{n}^\epsilon,\mathbf{R}^\epsilon,\mathbf{S}^\epsilon)(x)$ of the form
\begin{equation}\label{perb}
n_i^\epsilon(x)=n_i(x)+\epsilon v_i(x)+o(\epsilon), \quad{\rm and }\quad
\begin{cases}
R_i^\epsilon(x)=R_i(x)+\epsilon r_i(x)+o(\epsilon),\\
S_i^\epsilon(x)=S_i(x)+\epsilon s_i(x)+o(\epsilon),
\end{cases}
    \end{equation}
for $i=1,2,3$ and $\mathbf{n}^\epsilon=(n_1^\epsilon,n_2^\epsilon,n_3^\epsilon)$,  $\mathbf{R}^\epsilon=(R_1^\epsilon,R_2^\epsilon,R_3^\epsilon)$, $\mathbf{S}^\epsilon=(S_1^\epsilon,S_2^\epsilon,S_3^\epsilon)$.

Let the tangent vectors $\mathbf{r}=(r_1,r_2,r_3), \mathbf{s}=(s_1,s_2,s_3)$ be given, from \eqref{R-S-eqn}  and \eqref{perb}, it follows that the perturbation $\mathbf{v}=(v_1,v_2,v_3)$ can be uniquely determined by
\begin{equation}\label{vx}
\mathbf{v}_x=\frac{\mathbf{r}-\mathbf{s}}{2c(n_1)}-\frac{\mathbf{R}-\mathbf{S}}{2c^2(n_1)}c'(n_1)v_1,\qquad \mathbf{v}(t,0)=\mathbf{0},
\end{equation}
and
\begin{equation}\label{vt}
\mathbf{v}_t=(\mathbf{r}+\mathbf{s})/2.
\end{equation}
Moreover, in light of \eqref{vwl} and \eqref{R-S-eqn},  it is  straightforward to check that the first order perturbations $\mathbf{v},\mathbf{s},\mathbf{r}$ must satisfy the equations
\begin{equation}\label{vtt}
\begin{split}
&\partial_{tt}v_i-c^2 \partial_{xx}v_{i}=2\big[(c')^2\partial_x n_1\partial_x n_i+cc''\partial_x n_1\partial_x n_i+cc'\partial_{xx} n_i\big]v_1-\big[|\mathbf{n}_t|^2-(2c^2-\zeta_i)|\mathbf{n}_x|^2\big]v_i\\
&\qquad\qquad-2\big[\mathbf{n}_t\cdot\mathbf{v}_t
-(2c^2-\zeta_i)\mathbf{n}_x\cdot\mathbf{v}_x\big]n_i+4cc'n_iv_1|\mathbf{n}_x|^2+2cc'(\partial_x n_1\partial_x v_i+\partial_x n_i\partial_x v_1),
\end{split}\end{equation}
and
\begin{equation}\label{rt}
\begin{cases}
\displaystyle\partial_t r_i-c\partial_x r_i=c'v_1\partial_x R_i+\frac{c'\zeta_i}{2c^3}v_1 n_i\big(\mathbf{R}^2+\mathbf{S}^2-2\mathbf{R}\cdot\mathbf{S}\big)+\frac{cc''-(c')^2}{2c^2}(R_i-S_i)R_1v_1\\
\qquad\qquad\qquad\displaystyle+\frac{v_i}{4c^2}\big[(c^2-\zeta_i)(\mathbf{R}^2+\mathbf{S}^2)-2(3c^2-\zeta_i)\mathbf{R\cdot\mathbf{S}}\big]+\displaystyle\frac{c'}{2c}\big[(R_i-S_i)r_1+(r_i-s_i)R_1\big]\\
\qquad\qquad\qquad\displaystyle+\frac{n_i}{2c^2}\big[(c^2-\zeta_i)(\mathbf{R}\cdot\mathbf{r}+\mathbf{S}\cdot\mathbf{s})
-(3c^2-\zeta_i)(\mathbf{R}\cdot\mathbf{s}+\mathbf{S}\cdot\mathbf{r})\big],\\
\displaystyle\partial_t s_i+c\partial_x s_i=-c'v_1\partial_x S_i+\frac{c'\zeta_i}{2c^3}v_1 n_i\big(\mathbf{R}^2+\mathbf{S}^2-2\mathbf{R}\cdot\mathbf{S}\big)+\frac{cc''-(c')^2}{2c^2}(R_i-S_i)S_1v_1\\
\qquad\qquad\qquad\displaystyle+\frac{v_i}{4c^2}\big[(c^2-\zeta_i)(\mathbf{R}^2+\mathbf{S}^2)-2(3c^2-\zeta_i)\mathbf{R\cdot\mathbf{S}}\big]\displaystyle+\frac{c'}{2c}\big[(R_i-S_i)s_1+(r_i-s_i)S_1\big]\\
\qquad\qquad\qquad\displaystyle+\frac{n_i}{2c^2}\big[(c^2-\zeta_i)(\mathbf{R}\cdot\mathbf{r}+\mathbf{S}\cdot\mathbf{s})
-(3c^2-\zeta_i)(\mathbf{R}\cdot\mathbf{s}+\mathbf{S}\cdot\mathbf{r})\big],\\
\end{cases}\end{equation}
for $i=1,2,3$ and $\zeta_1=\ga, \zeta_2=\zeta_3=\al.$

To continue, one also needs to add quantities, named as $w(t,x), z(t,x)$, to measure the horizontal shifts, corresponding to backward and forward directions, respectively, which provide enough freedom for planar transports. Here we require $w(t,x)$ to  satisfy
\[
\epsilon w(t,x)+o(\epsilon)=x^\epsilon(t)-x(t),
\]
where $x^\epsilon(t)$ and $x(t)$ are two backward characteristics starting from initial points $x^\epsilon(0)$ and $x(0)$.
Similarly, the function $\epsilon z(t,x)$ measures the difference of two forward characteristics.
More precisely, we choose $w, z$ to be the solutions of the following system
\begin{equation}\label{wz}
\begin{cases}
\displaystyle  w_t-cw_x=-c'(v_1+w\partial_x n_1 ),\\
\displaystyle  z_t+cz_x=c'(v_1+z\partial_x n_1 ),\\
w(0,x)=w_0(x),\qquad z(0,x)=z_0(x).
\end{cases}\end{equation}

With the above preparation, we can now define a Finsler norm on the space of tangent vectors
  $(\mathbf{v},\mathbf{r},\mathbf{s})$ and the flow itself $(\mathbf{n}, \mathbf{R}, \mathbf{S})$ as
\begin{equation}\label{Finsler v}
\|(\mathbf{v},\mathbf{r},\mathbf{s})\|_{(\mathbf{n},\mathbf{R},\mathbf{S})}: = \inf_{\mathbf{v}, \mathbf{r^*}, \mathbf{s^*}, w, z} \|(\mathbf{v}, \mathbf{r^*},\mathbf{s^*}, w, z)\|_{(\mathbf{n},\mathbf{R},\mathbf{S})},
\end{equation}
where the infimum is taken over the set of vertical displacements $\mathbf{v},\mathbf{r^*}=(r^*_1,r^*_2,r^*_3),\mathbf{s^*}=(s^*_1,s^*_2,s^*_3)$ and horizontal shifts $w,z$ which satisfy equations \eqref{vx}, \eqref{vt}, \eqref{wz} and
\begin{equation}\label{rseq}
\begin{cases}
\displaystyle r^*_i=r_i+w\partial_xR_i+\frac{n_i}{8c^3}\big[(c^2-\zeta_i)\mathbf{S}^2
-2(3c^2-\zeta_i)\mathbf{R}\cdot\mathbf{S}\big](w-z)
-\frac{c'}{4c^2}(w-z)R_1S_i,\\
\displaystyle s^*_i=s_i+z\partial_xS_i+\frac{n_i}{8c^3}\big[(c^2-\zeta_i)\mathbf{R}^2
-2(3c^2-\zeta_i)\mathbf{R}\cdot\mathbf{S}\big](w-z)
-\frac{c'}{4c^2}(w-z)R_iS_1,
\end{cases}
\end{equation}
for $i=1,2,3$ and $ \zeta_1=\ga, \zeta_2=\zeta_3=\al.$ Next, the norm $\|(\mathbf{v}, \mathbf{r^*},\mathbf{s^*}, w, z)\|_{(\mathbf{n},\mathbf{R},\mathbf{S})}$ is defined as
\begin{equation}\label{norm1}
\begin{split}
&\ \|(\mathbf{v}, \mathbf{r^*},\mathbf{s^*}, w, z)\|_{(\mathbf{n},\mathbf{R},\mathbf{S})}\\
&:=~\kappa_0\int_\mathbb{R} \big[|w|\, \mathcal{V}^-+|z|\, \mathcal{V}^+\big]\,dx+\kappa_1\int_\mathbb{R}  \big[|w|(1+\mathbf{R}^2)\, \mathcal{V}^-
+|z|(1+\mathbf{S}^2)\, \mathcal{V}^+\big]\,dx\\
&\quad+\kappa_2\sum_{i=1}^3\int_\mathbb{R}  \Big|v_i+\frac{R_iw-S_iz}{2c}\Big|\big[(1+\mathbf{R}^2)\, \mathcal{V}^-
+(1+\mathbf{S}^2)\, \mathcal{V}^+\big]\,dx\\
&\quad+\kappa_3\int_\mathbb{R} \Big[ \Big|w_x+\frac{c'}{4c^2}(w-z)S_1\Big|\, \mathcal{V}^-+
\Big|z_x+\frac{c'}{4c^2}(w-z)R_1\Big|\, \mathcal{V}^+\Big]\,dx \\
&\quad+\kappa_4\sum_{i=1}^3\int_\mathbb{R}  \Big[\Big|r^*_i+R_i\big(w_x+\frac{c'}{4c^2}(w-z)S_1\big)\Big|\, \mathcal{V}^-
+\Big|s^*_i+S_i\big(z_x+\frac{c'}{4c^2}(w-z)R_1\big)\Big|\, \mathcal{V}^+\Big]\,dx \\
&\quad+\kappa_5\int_\mathbb{R} \Big[ \Big|2\mathbf{R}\cdot \mathbf{r^*}+\mathbf{R}^2w_x+\frac{c'}{4c^2}(w-z)\mathbf{R}^2S_1\Big|\, \mathcal{V}^-\\
&\qquad\qquad\qquad+
\Big|2\mathbf{S}\cdot \mathbf{s^*}+\mathbf{S}^2z_x+\frac{c'}{4c^2}(w-z)\mathbf{S}^2R_1\Big|\, \mathcal{V}^+\Big]\,dx \\
&=: \sum_{j=0}^5\kappa_j\big(\int_\mathbb{R}  J_j^-\, \mathcal{V}^-\,dx+\int_\mathbb{R}  J_j^+\, \mathcal{V}^+\,dx\big)=: \sum_{j=0}^5\kappa_j I_j,
\end{split}
\end{equation}
where $\kappa_j$, $j=0,1,\cdots,5$ are the constants to be determined later, and $I_j, J_j^-, J_j^+$ are the corresponding terms in the above equation. On the other hand, in view of \eqref{balance}, the forward or backward energy might increase during the wave interaction, although the total energy is conserved. To balance this possible energy increase, a pair of {\em  interaction potentials} $\mathcal{V}^+/\mathcal{V}^-$ for forward/backward directions need to be added in \eqref{norm1} as \begin{equation*}\label{W}
\mathcal{V}^-:=1+\int_{-\infty}^x \mathbf{S}^2(y)\,dy,\quad
\mathcal{V}^+:=1+\int^{+\infty}_x \mathbf{R}^2(y)\,dy.
\end{equation*}
Then it follows from \eqref{balance} that
\begin{equation}\label{Westimate}
\begin{cases}
\displaystyle\mathcal{V}^-_t-c\mathcal{V}^-_x=
-2c\mathbf{S}^2+\int_{-\infty}^x \big[\frac{c'}{2c}(R_1\mathbf{S}^2-\mathbf{R}^2S_1)\big]\,dy\leq
-2c_0\mathbf{S}^2+G(t),\\
\displaystyle\mathcal{V}^+_t+c\mathcal{V}^+_x=
-2c \mathbf{R}^2+\int^{+\infty}_x \big[\frac{c'}{2c}(\mathbf{R}^2S_1-R_1\mathbf{S}^2)\big]\,dy\leq
-2c_0\mathbf{R}^2+G(t),\\
\end{cases}
\end{equation}
with
$\displaystyle G(t):=\int_{-\infty}^{+\infty} \Big|\frac{c'}{2c}(\mathbf{R}^2S_1-R_1\mathbf{S}^2)\Big|\,dy.$
Moreover, with the aid of \cite{CCD}, we can see that
\begin{equation}\label{4.9}
\int_0^T G(t)\leq C_T,
\end{equation}
for some constant $C_T$ depending only on $T$ and the total energy.

\vspace{.2cm}

 Now we briefly explain how to obtain $J_j^-, j=0,1,\cdots,5$ in \eqref{norm1}. And $J_j^+$ is symmetric for forward waves.

\vspace{.2cm}

\paragraph{\bf (1)} $J_1^-$ measures [change in $x$]$\cdot (1+\mathbf{R}^2)$, where
\[[\hbox{change in }x]=\lim_{\epsilon\rightarrow 0}\epsilon^{-1}(x^\epsilon-x)=w(x).\]

The terms in $I_0$ are corresponding to the variation of $|x|$ with base measure  with density $1$, which are added for a technical purpose.

\vspace{.2cm}
\paragraph{\bf (2)} $J_2^-$ measures $\displaystyle\sum_{i=1}^3$ [change in $n_i$]$\cdot (1+\mathbf{R}^2)$, where
\begin{equation*}
\begin{split}
[\hbox{change in }n_i]=\lim_{\epsilon\rightarrow 0}\frac{n_i^\epsilon(x^\epsilon) - n_i(x)}{ \epsilon} &= v_i(x) + \partial_x n_i(x) w(x)\\
&= v_i(x) + \frac{R_i(x)-S_i(x)}{2c(n_1(x))}w(x)\\
&= v_i(x) + \frac{R_iw-S_iz}{2c}+\frac{z-w}{2c}S_i.
\end{split}\end{equation*}
Here the term $\frac{z-w}{2c}S_i$ on the above equation is just balanced with the {\bf relative shift} term.

Here we use this term to introduce how to calculate the relative shift term. By the third equation of $\eqref{R-S-eqn}$, we have
$
\partial_x n_i=\frac{R_i-S_i}{2c}.
$
That is, roughly speaking,
\beq\label{deltau}
\Delta n_i\approx\frac{\Delta x}{2c}(R_i-S_i)=\frac{z-w}{2c}(R_i-S_i).
\eeq
Here the $S_i$ term balances $\frac{z-w}{2c}S_i$. We omit the $R_i$ term since it is a lower order term.



\paragraph{\bf (3)} $J_4^-$ measures $\displaystyle\sum_{i=1}^3$ [change of base  measure with density $R_i$]. More precisely
\begin{equation*}
\begin{split}
&\quad \lim_{\epsilon\rightarrow 0}\frac{ R_i^\epsilon(x^\epsilon)\,dx^\epsilon- R_i(x)\,dx}{\epsilon}\\
&= \lim_{\epsilon\rightarrow 0}\frac{ \Big( R_i^\varepsilon(x^\epsilon) -R_i(x) \Big)\,dx^\epsilon + R_i(x)(dx^\epsilon-dx)}{\epsilon} \\
&=  \Big(r_i(x) + w(x)\partial_x R_i(x)+R_i(x)w_x(x)\Big)\,dx ,
\end{split}
\end{equation*}
which together with the relative shift term gives $J_4^-$. Here we add some subtle adjustments in the relative shift terms to take account of interactions between forward and backward waves using \eqref{R-S-eqn}. As mentioned before, this is a new term comparing to the metric in \cite{BC2015}.

$J_3^-$ measures [change of base  measure with density $1$],  which is added to close the estimate of time derivatives for $J_4^-$. This term is to some extend lower order term of $J_4^-$.
\vspace{.2cm}
\paragraph{\bf (4)} $J_5^-$ measures [change of base  measure with density $\mathbf{R}^2$], using the identity
\begin{equation*}
\begin{split}
\big(\mathbf{R}^\epsilon(x^\epsilon) \big)^2 & = \mathbf{R}^2(x^\epsilon) + 2\epsilon \mathbf{R}(x^\epsilon)\cdot \mathbf{r}(x^\epsilon) + o(\epsilon) \\
& = \mathbf{R}^2(x) + 2\epsilon w(x) \mathbf{R}(x)\cdot \mathbf{R}_{x}(x) + 2\epsilon \mathbf{R}(x)\cdot \mathbf{r}(x) + o(\epsilon),
\end{split}
\end{equation*}
we obtain that
\begin{equation}\label{exp1}
\begin{split}
\big( \mathbf{R}^\epsilon_x(x^\epsilon) \big)^2 dx^\epsilon - \mathbf{R}^2(x) dx
=\Big( 2\epsilon \mathbf{R}(x)\cdot \mathbf{R}_{x}(x)w(x)  + 2\epsilon \mathbf{R}(x) \cdot\mathbf{r}(x) +\epsilon \mathbf{R}^2(x)w_x(x)+o(\epsilon)\Big)\,dx.
\end{split}
\end{equation}
On the other hand, as in \eqref{deltau}, if the mass with density $\mathbf{S}^2$ is transported from $x$ to $x+\epsilon z(x)$, in view of \eqref{balance}, the relative shift between forward and backward waves will contribute
\begin{equation}\label{exp2}
\frac{c'}{2c}(\mathbf{R}^2S_1-R_1\mathbf{S}^2)\frac{z-w}{2c} \epsilon.
\end{equation}
Subtracting \eqref{exp2} from  \eqref{exp1} we have
\begin{equation}\label{exp3}
2 \mathbf{R}\cdot\mathbf{r} + 2 \mathbf{R}\cdot \mathbf{R}_{x}w  + \mathbf{R}^2w_x
+\frac{c'}{4c^2}(\mathbf{R}^2S_1-R_1\mathbf{S}^2)(w-z).
\end{equation}
By using $|\mathbf{n}|=1$ and $c'=\frac{\gamma-\alpha}{c}n_1$, so that $\mathbf{R}\cdot\mathbf{n}=0$, we further obtain
\begin{equation*}\label{exp4}
2 \mathbf{R}\cdot\mathbf{r}^* = 2 \mathbf{R}\cdot\mathbf{r} + 2 \mathbf{R}\cdot \mathbf{R}_{x}w  -\frac{c'}{4c^2}R_1\mathbf{S}^2(w-z).
\end{equation*}
This together with \eqref{exp3} gives the term $J_6^-$.

\vspace{.2cm}

Now we state the main result of this section, which is showing that the norm of tangent vectors defined in \eqref{Finsler v} satisfies a Gr\"{o}wnwall type inequality.

\begin{Lemma}\label{lem_est}
 Let $(\mathbf{n},\mathbf{R},\mathbf{S})(t,x)$ be a smooth solution to \eqref{vwl} and \eqref{R-S-eqn} for $t\in[0,T]$, with $T>0$ be given. Assume that the first order perturbations $(\mathbf{v},\mathbf{r},\mathbf{s})$ satisfy the corresponding equations \eqref{vtt}--\eqref{rt}. Then it follows that
\begin{equation}\label{normest}
\|(\mathbf{v},\mathbf{r},\mathbf{s})(t)\|_{(\mathbf{n},\mathbf{R},\mathbf{S})(t)}\leq C\|(\mathbf{v},\mathbf{r},\mathbf{s})(0)\|_{(\mathbf{n},\mathbf{R},\mathbf{S})(0)},
\end{equation}
with the constant $C$ depending only on the initial total energy and $T$.
\end{Lemma}

\begin{proof}
To achieve \eqref{normest}, it suffices to show that
\begin{equation}\label{est on w and z}
{d \over dt}\|(\mathbf{v}, \mathbf{r^*},\mathbf{s^*}, w,z)(t)\|_{(\mathbf{n},\mathbf{R},\mathbf{S})(t)} \leq a(t) \|(\mathbf{v}, \mathbf{r^*},\mathbf{s^*}, w,z)(t)\|_{(\mathbf{n},\mathbf{R},\mathbf{S})(t)},
\end{equation}
for  any $w, z$ and $\mathbf{r^*}, \mathbf{s^*}$ satisfying (\ref{wz}) and (\ref{rseq}), with a local integrable function $a(t)$.
In fact, by elaborate calculations on the time derivatives of all terms in \eqref{norm1}, we have
\begin{equation}\label{sum}\left.\begin{array}{l}
\displaystyle\frac{dI_k}{dt}\leq~
C\sum_{\ell\in {\mathcal F}^l_k}
\left(\int_\mathbb{R}  (1+|\mathbf{S}|)\,J_\ell^-\,{\mathcal V}^- \,dx+\int_\mathbb{R}
(1+|\mathbf{R}|)J_\ell^+\,\,{\mathcal V}^+ \,dx
\right) \\[2mm]
\displaystyle \qquad \qquad+C\sum_{\ell\in {\mathcal F}^h_k}
\left(\int_\mathbb{R} (1+ \mathbf{S}^2)\,J_\ell^-\,{\mathcal V}^- \,dx+\int_\mathbb{R}   (1+\mathbf{R}^2)\,J_\ell^+\,{\mathcal V}^+\, dx
\right)\\[2mm]
\displaystyle\qquad\qquad+G(t)I_k-c_0\left(\int_\mathbb{R}    \mathbf{S}^2\,J_k^-\,{\mathcal V}^- \,dx+\int_\mathbb{R}   \mathbf{R}^2\,J_k^+\,{\mathcal V}^+ \,dx
\right).
\end{array}\right.
\end{equation}
The detail calculation for \eqref{sum} can be found in Appendix \ref{app}. Here $\mathcal {F}^l_k,\mathcal {F}^h_k\subset\{0,1,2,\cdots,5\}$ are suitable sets of indices from the estimates  \eqref{I0est}, \eqref{I1est}, \eqref{I2est}, \eqref{I3est}, \eqref{I4est} and \eqref{I5est}, where a graphical summary of all the a priori estimate is illustrated in Fig. \ref{f:wa36}.
For example, by \eqref{I3est}, $\mathcal {F}^l_3=\{2,3,4\}$ and $\mathcal {F}^h_3=\{1\}$.
Throughout the paper, $C>0$ is a generic constant depending only on the initial total energy and $T$, which may vary in different estimates.

\begin{figure}[htbp]
   \centering
\includegraphics[scale=0.6]{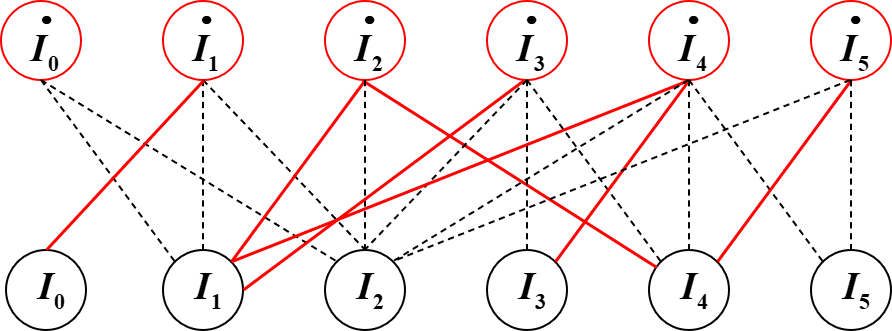}
\caption{ $\dot{I_k}=\frac{d I_k}{dt}$. If $\ell\in {\mathcal F}_k^l$, then $\dot{I_k}$ and $I_\ell$ are connected by a dash line. If $\ell\in {\mathcal F}_h^l$, then $\dot{I_k}$ and $I_\ell$ are connected by a solid line. $k\rightarrow {\mathcal F}_k^h \subset\{0,1,\cdots 5\}$ has {\bf{no cycle!}}  Choose $\kappa_k$ in a certain order
($\kappa_0\gg\kappa_1\gg\kappa_3\gg\kappa_4\gg \kappa_2, \kappa_5
$) to prove \eqref{est on w and z}. }
   \label{f:wa36}
\end{figure}

Since there is no cycle for the relation tree ${\mathcal F}_k^h$,
we can choose a suitable small constant $\delta>0$, with the weighted norm defined by
\begin{equation*}
\|(\mathbf{v}, \mathbf{r^*},\mathbf{s^*}, w,z)(t)\|_{(\mathbf{n},\mathbf{R},\mathbf{S})(t)} :=I_0+\delta I_1+\delta^4 I_2+\delta^2 I_3+\delta^3 I_4+\delta^4 I_5,
\end{equation*}
we arrive at the desired estimate \eqref{est on w and z}. Therefore, the proof of Lemma \ref{lem_est} is finished.
\end{proof}

\section{Metric for piecewise smooth solutions}
Having constructed a weighted norm on tangent vectors for smooth solutions, our main goal now is how to extend this metric to general weak solutions.
By the strong nonlinearity of equations, solutions with smooth initial data can lose regularity in finite time. When this happens, the tangent vector ${\bf v}$ may no longer exist since there may be no regular path between two solutions. Even if the tangent vector does exist, it is not obvious that the estimate in Lemma \ref{lem_est} holds.

In this section, we first extend the metric to piecewise smooth solutions.

A natural question arises as whether there are a dense set of
piecewise smooth paths of solutions, whose weighted length can be controlled in time. We note that an analogous theorem proved in \cite{CCD} by authors gives a positive answer to this question. Roughly speaking, we proved that, for generic smooth initial data, the solution is piecewise smooth. Its gradient blows up along finitely many smooth curves in the $t$-$x$ plane. In subsection \ref{sec:4.1},  we first review this basic construction and the characterization of generic singularities \cite{CCD}.

\subsection{Generic regularity and smooth path of solutions}\label{sec:4.1}
We define the forward and backward characteristics as follows:
\begin{equation*}
\begin{cases}
\frac{d}{ds}x^\pm(s,t,x)=\pm c(n_1(s,x^\pm(s,t,x))),\\
x^\pm|_{s=t}=x.
\end{cases}
\end{equation*}
Then we define the coordinate transformation $(t,x)\to (X,Y)$ where
\begin{equation*}
X~:=~\int_0^{x^-(0,t,x)}[1+\mathbf R^2(0,y)]\,dy, \quad \text{and }\quad Y~:=~\int^0_{x^+(0,t,x)}[1+\mathbf S^2(0,y)]\,dy.
\end{equation*}
Of course this implies
\begin{equation}\label{XY}
X_t-c(n_1)X_x=0,\quad Y_t+c(n_1)Y_x=0.
\end{equation}
Furthermore, for any smooth function $f$, by using \eqref{XY}, we obtain  that
\begin{equation}\label{2.15}
\begin{cases}
f_t+c(n_1)f_x=(X_t+c(n_1)X_x)f_X=2c(n_1)X_x f_X,\\
f_t-c(n_1)f_x=(Y_t-c(n_1)Y_x)f_Y=-2c(n_1)Y_x f_Y.
\end{cases}
\end{equation}

Now, we choose new variables to avoid the blowup as:
\begin{eqnarray*}
   &&\displaystyle p=\frac{1+|\mathbf R|^2}{X_x},\qquad q=\frac{1+|\mathbf S|^2}{-Y_x}, \nonumber \\
   &&\displaystyle\mathbf{L}=(l_1,l_2,l_3)=\frac{\mathbf R}{1+|\mathbf R|^2},\quad \mathbf m=(m_1,m_2,m_3)=\frac{\mathbf S}{1+|\mathbf S|^2},  \nonumber\\
  &&\displaystyle h_1=\frac{1}{1+|\mathbf R|^2}, \quad h_2=\frac{1}{1+|\mathbf S|^2}.\nonumber
\end{eqnarray*}
With the above notations, these variables satisfy the following semi-linear system, c.f. \cite{CZZ}
\begin{equation}\label{2.17}
\begin{cases}
&\partial_Y l_i=\displaystyle\frac{q}{8c^3(n_1)}[(c^2(n_1)-\zeta_i)(h_1+h_2-2h_1h_2)-2(3c^2(n_1)-\zeta_i)\mathbf L\cdot \mathbf m]n_i\\
&\displaystyle\qquad\qquad+\frac{c'(n_1)}{4c^2(n_1)}l_1q(l_i-m_i),\\
&\partial_X m_i =\displaystyle\frac{p}{8c^3(n_1)}[(c^2(n_1)-\zeta_i)(h_1+h_2-2h_1h_2)-2(3c^2(n_1)-\zeta_i)\mathbf L\cdot \mathbf m]n_i\\
&\displaystyle\qquad\qquad-\frac{c'(n_1)}{4c^2(n_1)}m_1p(l_i-m_i),\\
&\partial_Y\mathbf n =\displaystyle \frac{q}{2c(n_1)}\mathbf m, \qquad (\text {or } \quad\partial_X\mathbf n =\frac{p}{2c(n_1)}\mathbf L ),\\
&\partial_Y h_1 =\displaystyle \frac{c'(n_1)}{4c^2(n_1)}ql_1(h_1- h_2),\qquad \partial_X h_2 = \frac{c'(n_1)}{4c^2(n_1)}pm_1(h_2- h_1),\\
&p_Y =\displaystyle -\frac{c'(n_1)}{4c^2(n_1)}pq(l_1 -m_1),\qquad q_X = \frac{c'(n_1)}{4c^2(n_1)}pq(l_1 -m_1),
\end{cases}
\end{equation}
with $i=1,2,3$, $\zeta_1=\gamma$ and $\zeta_2=\zeta_3=\alpha$. Using \eqref{2.15}, by letting $f=t$ or $x$, we obtain the equations
\begin{equation}\label{2.18}
t_X=\frac{ph_1}{2c(n_1)},\quad t_Y=\frac{qh_2}{2c(n_1)},\quad x_X=\frac{ph_1}{2},\quad x_Y=-\frac{qh_2}{2}.
\end{equation}
On the initial line $t=0$ in $(t,x)$ plane, we transform it to a particular curve
\begin{equation*}
\gamma_0=\{(X,Y);~X+Y=0\}\subset\mathbb{R}^2
\end{equation*}
in the $(X,Y)$ plane. Along the curve $\gamma_0$
parameterized by $x\mapsto(\bar{X}(x),\bar{Y}(x)):=(x,-x)$, we assign the boundary data $(\bar{\mathbf n},\bar{\mathbf L}, \bar{\mathbf m},\bar{h}_1, \bar{h}_2,\bar{p},\bar{q})$ defined by their definition evaluated at the initial data \eqref{ID}, that is
\begin{equation*}\label{2.19}
\begin{split}
&\bar{\mathbf n}=\mathbf n_0(x), \quad \bar{\mathbf L}=\mathbf R(0,x)\bar{h}_1, \quad\bar{\mathbf m}=\mathbf S(0,x)\bar{h}_2, \\
&\bar{h}_1=\frac{1}{1+|\mathbf R(0,x)|^2}, \quad \bar{h}_2=\frac{1}{1+|\mathbf S(0,x)|^2},\\
&\bar{p}=1+|\mathbf R(0,x)|^2,  \quad\bar{q}=1+|\mathbf S(0,x)|^2,
\end{split}
\end{equation*}
where
\begin{equation*}
\mathbf R(0,x)=\mathbf n_1+c(n_{10}(x))\mathbf n'_0(x), \quad
\mathbf S(0,x)=\mathbf n_1-c(n_{10}(x))\mathbf n'_0(x).
\end{equation*}

The existence and uniqueness of global weak energy conservative solutions of \eqref{vwl} has been established in \cite{CCD, CZZ}, by transforming the solution $\mathbf n(X,Y)$ of \eqref{2.17} to $\mathbf n(t, x)$ on the original variables $(t, x)$:
\begin{Lemma}[\cite{CCD, CZZ}]\label{Lemma 2.1}
Let the generic condition \eqref{gencon} and initial data \eqref{ID} be satisfied.
Then there exists a unique solution $(X,Y)\mapsto (\mathbf n, \mathbf L, \mathbf m, h_1, h_2, p, q, x,t)(X,Y)$ with $p,q>0$ to the system \eqref{2.17}--\eqref{2.18} with boundary data assigned along the line $\gamma_0$. Moreover the set of points
\begin{equation}\label{2.20}
\big\{(t(X,Y),x(X,Y),\mathbf n(X,Y));~(X,Y)\in \mathbb{R}^2\}
\end{equation}
is the graph of a unique conservative solution $\mathbf n=\mathbf n(X,Y)$ to the Cauchy problem \eqref{vwl}--\eqref{ID}.
\end{Lemma}


To continue, we introduce the following definitions.

\begin{Definition}\label{def_gensin}
A solution $\mathbf n=\mathbf n(x,t)$ of \eqref{vwl} is called has {\bf generic singularities} for $t\in[0,T]$ if it admits a representation of the form \eqref{2.20}, where

{\rm (i)} the functions $(\mathbf n, \mathbf L, \mathbf m, h_1, h_2, p, q, x,t)(X,Y)$ are $\mathcal{C}^\infty$,

{\rm (ii)} for $t(X,Y)\in[0,T]$, the following generic conditions hold:
\begin{equation*}\label{generic_con}
\begin{cases}
 h_1=0, \mathbf L_X=\mathbf{0} \Longrightarrow \mathbf L_Y\neq \mathbf{0},\mathbf L_{XX}\neq \mathbf{0},\\
h_2=0, \mathbf{m}_Y=\mathbf{0} \Longrightarrow \mathbf{m}_X\neq \mathbf{0},\mathbf{m}_{YY}\neq \mathbf{0},\\
 h_1=0, h_2=0 \Longrightarrow \mathbf{L}_X\neq\mathbf{0},  \mathbf{m}_Y\neq \mathbf{0}.\\
\end{cases}\end{equation*}
\end{Definition}


\begin{Definition}\label{def_repath}
A path of initial data $\Gamma^0:\lambda\mapsto (\mathbf{n}_0^\lambda,\mathbf{n}_1^\lambda)$, $\lambda\in[0,1]$ is called a {\bf piecewise regular path} if the following conditions hold.

{\rm (i)} There exists a continuous map $(X,Y,\lambda)\mapsto (\mathbf n, \mathbf L, \mathbf m, h_1, h_2, p, q, x,t)$ such that the semilinear system \eqref{2.17}--\eqref{2.18} holds for $\lambda\in[0,1]$, and the function $\mathbf{n}^\lambda(x,t)$ whose graph is
\begin{equation*}
\text{ Graph }(\mathbf n^\lambda)=\{(x,t,\mathbf n)(X,Y,\lambda);~(X,Y)\in \mathbb{R}^2\}
\end{equation*}
provides the conservation solution of \eqref{vwl} with initial data $\mathbf n^\lambda(x,0)=\mathbf n^\lambda_0(x),\mathbf n_t^\lambda(x,0)=\mathbf n_1^\lambda(x)$.

{\rm (ii)} There exist finitely many values $0=\lambda_0<\lambda_1<\cdots<\lambda_N=1$ such that the map $(X,Y,\lambda)\mapsto(\mathbf n, \mathbf L, \mathbf m, h_1, h_2, p, q, x,t)$ is $\mathcal{C}^\infty$ for $\lambda\in(\lambda_{i-1},\lambda_i), i=1,\cdots, N$, and the solution $\mathbf n^\lambda=\mathbf n^\lambda(x,t)$ has only generic singularities at time $t=0$.

In addition, if for all $\lambda\in[0,1]\backslash\{\lambda_1,\cdots,\lambda_N\}$, the solution $\mathbf n^\lambda$ has only generic singularities for $t\in [0,T]$, then we say that the path of solution $\Gamma^t: \lambda\mapsto (\mathbf n^\lambda,\mathbf n^\lambda_t)$ is {\bf piecewise regular} for $t\in[0,T]$.
\end{Definition}

The following result shows that the set of piecewise regular paths is dense.
\begin{Corollary}\label{thm_repath}
Assume the generic condition \eqref{gencon} holds.
For any fixed $T>0,$ let $\lambda\mapsto(\mathbf n^\lambda, \mathbf L^\lambda, \mathbf m^\lambda,$ $h_1^\lambda, h_2^\lambda, p^\lambda, q^\lambda, x^\lambda,t^\lambda), \lambda\in[0,1],$ be a smooth path of solutions to the system \eqref{2.17}--\eqref{2.18}. Then there exists a sequence of paths of solutions $\lambda\mapsto(\mathbf n^\lambda_i, \mathbf L^\lambda_i, \mathbf m^\lambda_i,(h_1^\lambda)_i, (h_2^\lambda)_i, p^\lambda_i, q^\lambda_i, x^\lambda_i,$ $t^\lambda_i),$ such that

{\rm (i)} For each $i\geq 1$, the path of the corresponding solution of \eqref{vwl} $\lambda\mapsto \mathbf{n}_i^\lambda$ is regular for $t\in[0,T]$ in the sense of Definition \ref{def_repath}.

{\rm (ii)} For any bounded domain $\Sigma$ in the $(X$,$Y)$ space, the functions $(\mathbf n^\lambda_i, \mathbf L^\lambda_i, \mathbf m^\lambda_i,(h_1^\lambda)_i, (h_2^\lambda)_i, p^\lambda_i,$ $ q^\lambda_i, x^\lambda_i, t^\lambda_i)$ converge to $(\mathbf n^\lambda, \mathbf L^\lambda, \mathbf m^\lambda, h_1^\lambda, h_2^\lambda, p^\lambda, q^\lambda, x^\lambda,t^\lambda)$ uniformly in $\mathcal{C}^k([0,1]\times\Sigma)$, for every $k\geq 1$, as $i\to \infty.$
\end{Corollary}
The proof of this lemma is very similar to the corresponding one in
\cite{BC,CCD}. So we omit the proof and refer the readers to \cite{BC,CCD} for more details.

\subsection{Tangent vectors in transformed coordinates}
Now, we derive an expression for the norm of tangent
vectors \eqref{norm1} as a line integral in $X$-$Y$ coordinates.

For a reference solution $\mathbf{n}(x,t)$ of \eqref{vwl}, and let $\mathbf{n}^\varepsilon(x,t)$ be a family of perturbed solutions. In the $(X$,$Y)$ plane, denote $(\mathbf{n}, \mathbf{L}, \mathbf{m}, h_1, h_2, p, q, x,t)$ and $( \mathbf{n}^\varepsilon, \mathbf{L}^\varepsilon, \mathbf{m}^\varepsilon, h_1^\varepsilon, h_2^\varepsilon, p^\varepsilon, q^\varepsilon, x^\varepsilon, t^\varepsilon)$ be the corresponding smooth solutions of \eqref{2.17}--\eqref{2.18}, and consider the perturbed solutions of the form
\begin{equation*}
( n^\varepsilon_i, l^\varepsilon_i, m^\varepsilon_i, h_1^\varepsilon, h_2^\varepsilon, p^\varepsilon, q^\varepsilon, x^\varepsilon, t^\varepsilon)=(n_i, l_i, m_i, h_1, h_2, p, q, x,t)+\varepsilon(N_i, L_i, M_i,H_1, H_2, P, Q,\mathcal{X},\mathcal{T})+o(\varepsilon).
\end{equation*}
with $i=1,2,3, \mathbf{N}=(N_1,N_2,N_3), \mathbf{M}=(M_1,M_2,M_3),$ and $\mathcal{L}=(L_1,L_2,L_3)$. Here we denote the curve in ($X,Y)$ plane by
\begin{equation*}\label{curve}
\Lambda_\tau=\{(X,Y)\,|\,t(X,Y)=\tau\}=\{(X,Y(\tau, X));X\in\mathbb{R}\}=\{(X(\tau, Y),Y);Y\in\mathbb{R}\}\end{equation*}
and the perturbed curve as
 $$ \Lambda_\tau^\varepsilon=\{(X,Y)\,|\,t^\varepsilon(X,Y)=\tau\}=\{(X,Y^\varepsilon(\tau, X));X\in\mathbb{R}\}=\{(X^\varepsilon(\tau, Y),Y);Y\in\mathbb{R}\}.$$
By the smooth coefficients of system \eqref{2.17}--\eqref{2.18}, we see that the first order perturbations are well defined for $(X,Y)\in\mathbb{R}^2$ and also satisfy a linearized system.
In what follows, we express the terms $I_0$--$I_5$ of \eqref{norm1} in terms of $(\mathbf{N}, \mathcal{L}, \mathbf{M}, H_1, H_2, P, Q,\mathcal{X},\mathcal{T})$. First, we observe that
$$t^\varepsilon\big(X,Y^\varepsilon(\tau,X)\big)
=t^\varepsilon\big(X^\varepsilon(\tau,Y),Y\big)=\tau.$$
By the implicit function theorem, at $\varepsilon=0$, it holds that
\begin{equation}\label{xvar0}
\frac{\partial X^\varepsilon}{\partial\varepsilon}\Big|_{\varepsilon=0}
=-\mathcal{T}\frac{2c}{ph_1},
\quad {\rm and}\quad
\frac{\partial Y^\varepsilon}{\partial\varepsilon}\Big|_{\varepsilon=0}=-\mathcal{T}\frac{2c}{qh_2}.
\end{equation}

(1). The change in $x$ is compute by

\begin{equation}\label{wXchange}
\begin{split}
w&=\displaystyle\lim_{\varepsilon\to 0}\frac{x^\varepsilon\big(X,Y^\varepsilon(\tau,X)\big)
-x\big(X,Y(\tau,X)\big)}{\varepsilon}\\
&= \mathcal{X}\big(X,Y(\tau,X)\big)+x_Y\cdot \frac{\partial Y^\varepsilon}{\partial \varepsilon}\Big|_{\varepsilon=0}
=\big(\mathcal{X}+c\mathcal{T}\big)(X,Y(\tau,X)).
\end{split}
\end{equation}
Similar to \eqref{wXchange}, by \eqref{xvar0}, we have
\begin{equation}\label{zYchange}
\begin{split}
z&=\displaystyle\lim_{\varepsilon\to 0}\frac{x^\varepsilon\big(X^\varepsilon(\tau,Y),Y\big)
-x\big(X(\tau,Y),Y\big)}{\varepsilon}\\
&= \mathcal{X}\big(X(\tau,Y),Y\big)+x_X\cdot \frac{\partial X^\varepsilon}{\partial \varepsilon}\Big|_{\varepsilon=0}
=\big(\mathcal{X}-c\mathcal{T}\big)(X(\tau,Y),Y).
\end{split}
\end{equation}

(2). To see the change in $n_i$, observe that \eqref{2.17} and \eqref{xvar0} implies
\begin{equation*}
\begin{split}
v+\partial_x n_i w&=\displaystyle\frac{d}{d\varepsilon}n_i^\varepsilon
\big(X,Y^\varepsilon(\tau,X)\big)\Big|_{\varepsilon=0}\\
&= N_i\big(X,Y(\tau,X)\big)+\partial_Yn_i\cdot \frac{\partial Y^\varepsilon}{\partial \varepsilon}\Big|_{\varepsilon=0}
=\big(N_i-\frac{\mathcal{T} m_i}{h_2}\big)(X,Y(\tau,X)),
\end{split}
\end{equation*}
with $i=1,2,3,$ so that
\begin{equation}\label{uXchange}
v_i+\frac{R_iw-s_iz}{2c}=v_i+\partial_x n_i w+\frac{w-z}{2c}S_i=N_i(X,Y(\tau,X)),
\end{equation}
where we have used \eqref{R-S-eqn}.

(3). Now we want to estimate the change in the base measure with density $1+\mathbf{R}^2$.  By \eqref{2.17}, we obtain
\begin{equation}\label{pXchange}
\frac{d}{d\varepsilon}p^\varepsilon\big(X,Y^\varepsilon(\tau,X)\big)\Big|_{\varepsilon=0}
= P+p_Y\cdot \frac{\partial Y^\varepsilon}{\partial \varepsilon}\Big|_{\varepsilon=0}=P+\frac{c'\mathcal{T}p }{2c h_2}(l_1-m_1).
\end{equation}
Applying \eqref{2.17} again, we have the change in base measure with density $\mathbf{R}^2$:
\begin{equation}\label{R2change}
\begin{split}
&\frac{d}{d\varepsilon}\Big(\big(p^\varepsilon(1-h_1^\varepsilon)\big)\big(X,Y^\varepsilon(\tau,X)\big)\Big)\Big|_{\varepsilon=0}\\
&= \big(P+p_Y\cdot \frac{\partial Y^\varepsilon}{\partial \varepsilon}\Big|_{\varepsilon=0}\big)(1-h_1)-p\big(H_1+\partial_Yh_1\cdot \frac{\partial Y^\varepsilon}{\partial \varepsilon}\Big|_{\varepsilon=0}\big)\\
&=P(1-h_1)-H_1p+\frac{c'\mathcal{T}p}{2c}\Big(\frac{l_1-m_1}{h_2}+\frac{m_1h_1}{h_2}-l_1\Big).
\end{split}\end{equation}
As an immediate consequence of \eqref{pXchange} and \eqref{R2change}, we can achieve the change in base measure with density 1 by subtracting \eqref{pXchange} from \eqref{R2change}:
\begin{equation}\label{1change}
\begin{split}
h_1P+H_1p++\frac{c'\mathcal{T}p}{2c}\Big(l_1-\frac{m_1h_1}{h_2}\Big).
\end{split}\end{equation}

(4). Finally, for the change in the base measure with density $R_i$, it follows from \eqref{2.17}  and  \eqref{xvar0} that
\begin{equation*}
\begin{split}
&\quad \frac{d}{d\varepsilon}\Big(p^\varepsilon l^\varepsilon_i \big(X,Y^\varepsilon(\tau,X)\big)\Big)\Big|_{\varepsilon=0}\\
&= p\big(L_i+\partial_Yl_i\cdot \frac{\partial Y^\varepsilon}{\partial \varepsilon}\Big|_{\varepsilon=0}\big)+l_i\big(P+\partial_Y p\cdot \frac{\partial Y^\varepsilon}{\partial \varepsilon}\Big|_{\varepsilon=0}\big)\\
&=pL_i+l_iP-\frac{\mathcal{T}n_ip}{4c^2 h_2}[(c^2-\zeta_i)(h_1+h_2-2h_1h_2)-2(3c^2-\zeta_i)\mathbf{L}\cdot\mathbf{m}]
+\frac{c'\mathcal{T}p}{2ch_2}(l_1m_i-l_im_1),
\end{split}\end{equation*}
with $i=1,2,3$, $\zeta_1=\gamma$ and $\zeta_2=\zeta_3=\alpha$.

Notice that

$$(1+\mathbf{R}^2)dx=pdX,\quad (1+\mathbf{S}^2)dx=-qdY.$$
Base on the estimate \eqref{wXchange}--\eqref{1change}, we deduce that the weighted norm \eqref{norm1} can be rewritten as a line integral over the line $\Lambda_\tau$ defined in \eqref{curve}:
\begin{equation}\label{normXY}
 \|(\mathbf{v}, \mathbf{r^*},\mathbf{s^*}, w, z)\|_{(\mathbf{n},\mathbf{R},\mathbf{S})}=\sum_{j=0}^5\kappa_j\int_{\Lambda_\tau}
\Big(|J_j^-|\mathcal{V}^-dX+|J_j^+|\mathcal{V}^+dY\Big),
\end{equation}
where
\begin{eqnarray*}
J_0^-&=&\big(\mathcal{X}+c\mathcal{T}\big)ph_1,\\
J_1^-&=&\big(\mathcal{X}+c\mathcal{T}\big)p,\\
J_2^-&=&\sum_{i=1}^3N_ip,\\
J_3^-&=&h_1P+pH_1+\frac{c'\mathcal{T}p}{2c }l_1,\\
J_4^-&=&\sum_{i=1}^3 \Big(l_iP+pL_i-\frac{n_i\mathcal{T}p}{4c^2}(c^2-\zeta_i)(1-h_1)
\Big),\\
J_5^-&=&(1-h_1) P-pH_1,\\
\end{eqnarray*}
and
\begin{eqnarray*}
J_0^+&=&\big(\mathcal{X}-c\mathcal{T}\big)qh_2,\\
J_1^+&=&\big(\mathcal{X}-c\mathcal{T}\big)q,\\
J_2^+&=&\sum_{i=1}^3N_iq,\\
J_3^+&=&h_2Q+qH_2+\frac{c'\mathcal{T}q}{2c }m_1,\\
J_4^+&=&\sum_{i=1}^3 \Big(m_iQ+qM_i-\frac{n_i\mathcal{T}q}{4c^2}(c^2-\zeta_i)(1-h_2)
\Big),\\
J_5^+&=&(1-h_2)Q-qH_2.\\
\end{eqnarray*}
Furthermore, it is straightforward to check that the integrands $J_j^\pm$ are all smooth, for $j=0,1,\cdots,5$.

\subsection{Length of piecewise regular paths}\label{sub_piecewise}
Now we define the weighted length of a piecewise regular path.

\begin{Definition}\label{def_piece}
The length $\|\Gamma^t\|$ of the piecewise regular path $\Gamma^t: \lambda\mapsto \big(\mathbf{n}^\lambda(t),\mathbf{n}^\lambda_t(t)\big)$ is defined as
\begin{equation}\label{piecedef}
    \|\Gamma^t\|=\inf_{\Gamma^t}\int_0^1\Big\{\sum_{j=0}^5\kappa_j\int_{\Lambda_t^\lambda}
\Big(|(J_j^-)^\lambda|\mathcal{V}^-dX+|(J_j^+)^\lambda|\mathcal{V}^+dY\Big)\Big\}\,d\lambda,
\end{equation}
where the infimum is taken over all piecewise smooth relabelings of the $X$-$Y$ coordinates and $\Lambda_\tau^\lambda:=\{(X,Y);t^\lambda(X,Y)=\tau\}$.
\end{Definition}

Then we state the main theorem of this section, which implies  the appearance of the generic singularity will not impact the Lipschitz property of this metric.
\begin{Theorem}\label{thm_length}
 Let $T>0$ be given, consider a path of solutions $\lambda\mapsto \big(\mathbf{n}^\lambda(t),\mathbf{n}^\lambda_t(t)\big)$ of \eqref{vwl}, which is piecewise regular for $t\in[0,T]$. Moreover, the total energy is less than some constants $E>0$. Then there exists constants $\kappa_0,\kappa_1,\cdots,\kappa_5$ in \eqref{piecedef} and $C>0$, such that for any $0\leq t\leq T$, it holds
\begin{equation*}\label{pieceest}
\|\Gamma^t\|\leq C\|\Gamma^0\|,
\end{equation*}
where the constant $C$ depends only on $T$ and $E$.
\end{Theorem}

The proof of this lemma is similar to \cite{BC2015}, and we omit it here for brevity.

\section{Construction of the geodesic distance for general weak solutions}
Our final goal is to construct a geodesic distance, under which the general weak solutions obtained in Theorem \ref{CZZthm}--\ref{ECthm} is Lipschitz continuous.

The main idea is to extend the metric from generic piecewise smooth solution in Theorem \ref{thm_length} to general weak solution by taking limit from generic solutions.
To this end, we would like to point out  the generic regularity theorem in \cite{CCD}, that is, there exists an open dense set of initial data $\mathcal{D}\subset \Big(\mathcal{C}^3(\mathbb{R})\cap H^1(\mathbb{R})\big)\times \big(\mathcal{C}^2(\mathbb{R})\cap L^2(\mathbb{R})\big)$, such that, for $(n_{i0},n_{i1})\in\mathcal{D}$, the conservative solution $\mathbf{n}=(n_1,n_2,n_3)$ of \eqref{vwl} has only generic singularities. The structure of conservative solution  thus provides the ideal tool to construct a distance on a set $$\mathcal{D}^\infty:=\mathcal{C}_0^\infty\cap\mathcal{D},$$  by optimizing over all piecewise regular paths connecting two solutions of \eqref{vwl}. As a consequence, we extend our distance from space $\mathcal{D}^\infty$ to a larger domain by using the semilinear system \eqref{2.17}--\eqref{2.18} and Theorem \ref{thm_length},

We first introduce some definitions. For future use, we begin by introducing the subset of all data with energy  less than any fix constant $E>0$, specifically,
\begin{equation*}
\Omega:=\{(n_i,n_{it})\in H^1(\mathbb{R})\times L^{2}(\mathbb{R}); ~\mathcal{E}(\mathbf{n},\mathbf{n}_t):=\int_\mathbb{R}[\mathbf{n}_t^2+c(n_1)
\mathbf{n}_x^2]\,dx
\leq E\}.
\end{equation*}

\begin{Definition}\label{def_piecepath}
For solutions with initial data in $\mathcal{D}^\infty\cap~\Omega$, we define the geodesic distance $d\big((\mathbf{n},\mathbf{n}_t),$ $ (\hat{\mathbf{n}},\hat{\mathbf{n}}_t)\big)$ as
\begin{equation*}
\begin{split}
d\big((\mathbf{n},\mathbf{n}_t),(\hat{\mathbf{n}},\hat{\mathbf{n}}_t)\big):=\inf \{ \|\Gamma^t\|: &\Gamma^t \text{ is a piecewise regular path}, \Gamma^t(0)=(\mathbf{n},\mathbf{n}_t),\\
& \Gamma^t(1)=(\hat{\mathbf{n}},\hat{\mathbf{n}}_t), \mathcal{E}(\mathbf{n}^\lambda,\mathbf{n}_t^\lambda)\leq E, \text{ for all } \lambda\in[0,1]\},
\end{split}\end{equation*}
for any time $t$, where the infimum is taken over all weighted lengths of piecewise regular paths $\lambda\mapsto (\mathbf{n}^\lambda,\mathbf{n}_t^\lambda)$, which connect $(\mathbf{n},\mathbf{n}_t)$ with $(\hat{\mathbf{n}},\hat{\mathbf{n}}_t)$.
\end{Definition}

The definition $J(\cdot,\cdot)$ actually defines a distance because  after a suitable re-parameterization, the concatenation of two piecewise regular paths is still a piecewise regular path. With this definition of distance, the distance for general weak solutions is defined as follows.

\begin{Definition}\label{def_weak}
Let $(\mathbf{n}_0,\mathbf{n}_1)$ and $(\hat{\mathbf{n}}_0,\hat{\mathbf{n}}_1)$ be two initial data as required in the existence and uniqueness Theorem \ref{CZZthm} and Theorem \ref{ECthm}. Denote $\mathbf{n}$ and $\hat{\mathbf{n}}$ to be the corresponding global weak solutions, then for any time $t$, we define,
\begin{equation*}
d\big((\mathbf{n},\mathbf{n}_t),(\hat{\mathbf{n}},\hat{\mathbf{n}}_t)\big):=\lim_{k\rightarrow \infty}d\big((\mathbf{n}^k,\mathbf{n}_t^k),(\hat{\mathbf{n}}^k,\hat{\mathbf{n}}_t^k)\big),
\end{equation*}
for any two sequences of solutions $(\mathbf{n}^k,\mathbf{n}_t^k)$ and $(\hat{\mathbf{n}}^k,\hat{\mathbf{n}}_t^k)$ with the corresponding initial data in $\mathcal{D}^\infty \cap \Omega$, moreover for $i=1,2,3$,
\[
\|(n^k_{i0}-n_{i0}, \hat{n}^k_{i0}-\hat{n}_{i0})\|_{H^1}\rightarrow 0,\quad\hbox{and}\quad
\|(n^k_{i1}-n_{i1}, \hat{n}^k_{i1}-\hat{n}_{i1})\|_{L^2}\rightarrow 0.
\]
\end{Definition}

Thanks to the analysis above, we now have all ingredients toward a proof of main theorem: Theorem \ref{thm_metric}.
First, we claim that the definition of this metric is well-defined. In fact, since the solution with initial data in $\mathcal{D}^\infty\cap \Omega$ is Lipschitz continuous, we thus derive that the limit in the definition \ref{def_weak} is independent on the selection of sequences. On the other hand, by the fact that $\mathcal{D}^\infty\cap \Omega$ is a dense set in the solution space,
one can easily extend the Lipschitz metric to the general solutions.
As a consequence of Theorem \ref{thm_length}, we deduce directly the result of Theorem \ref{thm_metric}.


\vskip 0.2cm
We end this section with the relations between our distance function $d$ with other distances
determined by various norms.

\begin{Proposition}[Comparison with the Sobolev metric]\label{Prop_sob}
For any two finite energy initial data $(\mathbf{n}_0,$ $\mathbf{n}_1)$ and $(\hat{\mathbf{n}}_0, \hat{\mathbf{n}}_1) \in \mathcal D^\infty$, one has
\begin{equation*}
d\big((\mathbf{n}_0,\mathbf{n}_1),(\hat{\mathbf{n}}_0, \hat{\mathbf{n}}_1)\big)\leq C\sum_{i=1}^3\big(\|n_{i0}-\hat{n}_{i0}\|_{H^1}+\|n_{i0}-\hat{n}_{i0}\|_{W^{1,1}}+
\|n_{i1}-\hat{n}_{i1}\|_{L^1}+\|n_{i1}-\hat{n}_{i1}\|_{L^2}\big),
\end{equation*}
with a constant $C$ depends only on the initial energy.
\end{Proposition}

\begin{proof}
In order to get an upper bound for this optimal transport metric, a natural choice is letting the shifts $w=z=0$ in \eqref{norm1}, in this way, the norm becomes
\begin{equation}
\label{norm_rec}
\begin{split}
&\quad\|(\mathbf{v}^\lambda, (\mathbf{r}^*)^\lambda, (\mathbf{s}^*)^\lambda, w^\lambda, z^\lambda)\|_{(\mathbf{n}^\lambda,\mathbf{R}^\lambda,\mathbf{S}^\lambda)}\\
&=\kappa_2\sum_{i=1}^3\int_\mathbb{R}  \big|v_i^\lambda\big|\big[(1+(\mathbf{R}^\lambda)^2)\, (\mathcal{V}^-)^\lambda+(1+(\mathbf{S}^\lambda)^2)\, (\mathcal{V}^+)^\lambda\big]\,dx\\
&\quad+\kappa_3\sum_{i=1}^3\int_\mathbb{R} \big[|r_i^\lambda|\, (\mathcal{V}^-)^\lambda
+|s_i^\lambda|\, (\mathcal{V}^+)^\lambda\big]\,dx \\
&\quad+\kappa_6\int_\mathbb{R} \Big[ \big|2\mathbf{R}^\lambda\cdot\mathbf{r}^\lambda\big|\, (\mathcal{V}^-)^\lambda+
\big|2\mathbf{S}^\lambda\cdot\mathbf{s}^\lambda\big|\, (\mathcal{V}^+)^\lambda\Big]\,dx.
\end{split}
\end{equation}
For $\lambda\in[0,1]$, consider the path $(\mathbf{n}^\lambda_0,\mathbf{n}^\lambda_1)$ connecting  $(\mathbf{n}_0,\mathbf{n}_1)$ and $(\hat{\mathbf{n}}_0, \hat{\mathbf{n}}_1)$, which satisfies
\begin{equation*}
\mathbf{R}^\lambda=\lambda \mathbf{R}+(1-\lambda)\hat{\mathbf{R}},\qquad
\mathbf{S}^\lambda=\lambda \mathbf{S}+(1-\lambda)\hat{\mathbf{S}}.
\end{equation*}
Indeed, by \eqref{u_rec}, one can easily verify that above equations recover a unique path $(\mathbf{n}^\lambda, \mathbf{n}^\lambda_t)$.
Moreover the energy $\int_\mathbb{R} (\mathbf{R}^\lambda)^2+(\mathbf{S}^\lambda)^2\,dx$ is bounded by the energies of $(\mathbf{n}_0,\mathbf{n}_1)$ and $(\hat{\mathbf{n}}_0, \hat{\mathbf{n}}_1)$.

To estimate the right hand side of \eqref{norm_rec}, we first observe that
\beq\label{r_rec}
\mathbf{r}^\lambda=\frac{d}{d\lambda} \mathbf{R}^\lambda=\mathbf{R}-\hat{\mathbf{R}}
,\quad {\rm and} \quad
\mathbf{s}^\lambda=\frac{d}{d\lambda} \mathbf{S}^\lambda=\mathbf{S}-\hat{\mathbf{S}}.
\eeq
Next,  from the definition of $\mathbf{R}, \mathbf{S}$ at \eqref{R-S}, it follows
\begin{equation}\label{u_rec}
\mathbf{n}^\lambda_x=\frac{\mathbf{R}^\lambda-\mathbf{S}^\lambda}{2c(n^\lambda_1)}.
\end{equation}
Since the right hand side is Lipschitz on $\mathbf{n}^\lambda$ and $\mathbf{n}^\lambda$ has compact support, one can easily prove the existence and uniqueness of $\mathbf{n}^\lambda(x)$. So $\mathbf{v}^\lambda=\frac{d}{d\lambda} \mathbf{n}^\lambda$ satisfies
\[
\mathbf{v}^\lambda_x=\frac{\mathbf{r}^\lambda-\mathbf{s}^\lambda}{2c(n_1^\lambda)}
-\frac{\mathbf{R}^\lambda-\mathbf{S}^\lambda}{2c^2(n_1^\lambda)}c'(n_1^\lambda)v^\lambda_1,
\]
whence, by virtue of \eqref{r_rec} and after a straightforward manipulation,
 we arrive at  estimate for $\mathbf{v}^\lambda$:
\beq\label{v_rec}
|v_i^\lambda|\leq K\, \sum_{i=1}^3(\| R_i-\hat {R}_i\|_{L^1}+\| S_i-\hat{S_i}\|_{L^1})
\eeq
for some constant $K$ and $i=1,2,3$.
Substituting \eqref{r_rec}--\eqref{v_rec} into \eqref{norm_rec}, we get the desired conclusion of Proposition \ref{Prop_sob} directly.
\end{proof}

Actually, thanks to our main Theorem \ref{thm_metric}, this proposition also tells that, for any $t\geq0$,
\begin{equation*}
d\big((\mathbf{n},\mathbf{n}_t)(t),(\hat{\mathbf{n}}, \hat{\mathbf{n}}_t)(t)\big)\leq C\sum_{i=1}^3\big(\|n_{i0}-\hat{n}_{i0}\|_{H^1}+\|n_{i0}-\hat{n}_{i0}\|_{W^{1,1}}+
\|n_{i1}-\hat{n}_{i1}\|_{L^1}+\|n_{i1}-\hat{n}_{i1}\|_{L^2}\big).
\end{equation*}


\begin{Proposition}[Comparison with $L^1$ metric]
\label{Prop_L1}
Let $\mathbf{n}(t),$ $\hat{\mathbf{n}}(t)$ be conservative solutions obtained in
 Theorem \ref{CZZthm} and Theorem \ref{ECthm} with initial data $n_{i0},\hat{n}_{i0}\in H^1(\mathbb{R})\cap L^1(\mathbb{R})$ and $n_{i1},\hat{n}_{i1}\in L^2(\mathbb{R})$, $i=1,2,3$, there exists some constant $C$ depends only on the upper bound for the total energy, such that,
\begin{equation*}\label{L1}
\sum_{i=1}^3\|n_i-\hat{n}_i\|_{L^1}\leq C\cdot d\Big((\mathbf{n},\mathbf{n}_t)(t),(\hat{\mathbf{n}}, \hat{\mathbf{n}}_t)(t)\Big).
\end{equation*}
\end{Proposition}

\begin{proof} Suppose that $\Gamma^t:\lambda\mapsto \big(\mathbf{n}^\lambda(t),\mathbf{n}^\lambda_t(t)\big)$ is a regular path connecting $\mathbf{n}(t)$ with $\hat{\mathbf{n}}(t)$. It is obvious that
\begin{equation}\label{L11}
\begin{split}
|v_i|&=\Big|v_i+\frac{R_iw-S_iz}{2c}-\frac{R_iw-S_iz}{2c}\Big|\leq \Big|v_i+\frac{R_iw-S_iz}{2c}\Big|+\Big|\frac{R_iw}{2c}\Big|
+\Big|\frac{-S_iz}{2c}\Big|\\
&\leq \Big|v_i+\frac{R_iw-S_iz}{2c}\Big|+\frac{|w|(1+\mathbf{R}^2)}{4c}
+\frac{|z|(1+\mathbf{S}^2)}{4c}.
\end{split}\end{equation}
Recalling the definition \ref{def_weak}, we conclude from \eqref{norm1} and \eqref{L11} that
\begin{equation*}
\begin{split}
d\big((\mathbf{n},\mathbf{n}_t)(t),(\hat{\mathbf{n}}, \hat{\mathbf{n}}_t)(t)\big)&\geq C\inf_{\Gamma^t}\int_0^1\sum_{i=1}^3\int_{\mathbb{R}} |v_i^\lambda|\,dx\,d\lambda\\
&=C\inf_{\Gamma^t}\int_0^1\sum_{i=1}^3\int_{\mathbb{R}} \big|\frac{d n_i^\lambda}{d\lambda}\big|\,dx\,d\lambda\\
&\geq C\sum_{i=1}^3\|n_i-\hat{n}_i\|_{L^1}.
\end{split}
\end{equation*}
The proof of Proposition \ref{Prop_L1} is complete.
\end{proof}

\begin{Proposition}[Comparison with the Kantorovich-Rubinstein metric]\label{Prop_KR}
Let $\mathbf{n}(t),$ $\hat{\mathbf{n}}(t)$ be conservative solutions obtained in
 Theorem \ref{CZZthm} and Theorem \ref{ECthm} with initial data $n_{i0},\hat{n}_{i0}\in H^1(\mathbb{R})$ and $n_{i1},\hat{n}_{i1}\in L^2(\mathbb{R})$, $i=1,2,3$, there exists some constant $C$ depends only on the upper bound for the total energy, such that,
\begin{equation}\label{KR}
\sup_{\|f\|_{\mathcal{C}^1}\leq 1}\left|\int f \,d\mu-\int f\,d\hat{\mu}\right|\leq C \cdot d\Big((\mathbf{n},\mathbf{n}_t)(t),(\hat{\mathbf{n}}, \hat{\mathbf{n}}_t)(t)\Big),
\end{equation}
where $\mu,\hat{\mu}$ are the measures with densities $ \mathbf{n}_t^2+c(n_1) \mathbf{n}_x^2$ and $\hat{\mathbf{n}}_t^2+c(\hat{n}_1) \hat{\mathbf{n}}_x^2$ with respect to the Lebesgue measure.
The metric (\ref{KR}) is usually called a Kantorovich-Rubinstein distance, which is equivalent to a Wasserstein distance by a duality theorem \cite{V}.
\end{Proposition}

\begin{proof}
Let $\Gamma^t:\lambda\mapsto \big(\mathbf{n}^\lambda(t),\mathbf{n}^\lambda_t(t)\big)$ be a regular path connecting $\mathbf{n}(t)$ with $\hat{\mathbf{n}}(t)$. For any function $f$ with $\|f\|_{\mathcal{C}^1}\leq 1$,  let  $\mu^\lambda$  be the measure with density $ (\mathbf{n}^\lambda_t)^2+c(n_1^\lambda) (\mathbf{n}^\lambda_x)^2=\frac{1}{2}((\mathbf{R}^\lambda)^2+(\mathbf{S}^\lambda)^2)$ with respect to the Lebesgue measure. Then a direct computation gives rise to
\begin{equation*}
\begin{split}
&\left|\int_0^1\frac{d}{d\lambda}\int f\,d\mu^\lambda\,d\lambda \right| \leq C\int_0^1\int_{\mathbb{R}} \Big(|w^\lambda|(1+(\mathbf{R}^\lambda)^2)+|z^\lambda|(1+(\mathbf{S}^\lambda)^2)\Big)\,dx\,d\lambda\\
&\quad+\int_0^1\int_{\mathbb{R}}|f|\cdot\Big|
2\mathbf{R}^\lambda\cdot(\mathbf{r}^\lambda+\mathbf{R}^\lambda_xw^\lambda)+(\mathbf{R}^\lambda)^2w_x^\lambda
+2\mathbf{S}^\lambda\cdot(\mathbf{s}^\lambda+\mathbf{S}^\lambda_xz^\lambda)+(\mathbf{S}^\lambda)^2z_x^\lambda\Big|\,dx\,d\lambda\\
 &\leq C\int_0^1\int_{\mathbb{R}} \Big\{ |w^\lambda|(1+(\mathbf{R}^\lambda)^2)+|z^\lambda|(1+(\mathbf{S}^\lambda)^2)\\
&\quad\qquad\qquad+\big|2\mathbf{R}^\lambda\cdot(\mathbf{r}^*)^\lambda+(\mathbf{R}^\lambda)^2w_x^\lambda
+\frac{c'(w^\lambda-z^\lambda)}{4c^2}(\mathbf{R}^\lambda)^2S_1^\lambda\big|
\\
&\quad\qquad\qquad+\big|2\mathbf{S}^\lambda(\mathbf{s}^*)^\lambda+(\mathbf{S}^\lambda)^2z^\lambda_x+\frac{c'(w^\lambda-z^\lambda)}{4c^2}
 (\mathbf{S}^\lambda)^2R_1^\lambda\big|\Big\}\,dx\,d\lambda.
\end{split}\end{equation*}
This yields the desired conclusion of Proposition \ref{Prop_KR} immediately.
\end{proof}

\begin{appendices}


\section{The proof of Lemma \ref{lem_est}} \label{app}
Now we give detail proof for \eqref{sum}. This is the key estimate in the proof of Lemma \ref{lem_est}. We leave it in the appendix since it is lengthy.
\begin{proof}
The goal of the forthcoming computations is to validate the estimate \eqref{sum}.
 For the sake of clarity, we divide it into six steps.
\vspace{.2cm}
\paragraph{\bf Step 1.} To estimate the time derivative of $I_0$, we first write the first equation of \eqref{wz} in the form
\begin{equation*}\label{wt}
\displaystyle w_t-(cw)_x=-c'(v_1+
\frac{R_1w-S_1z}{2c})+\frac{c'}{c}S_1w
-\frac{c'}{2c}(S_1z+R_1w).
\end{equation*}
Now, from the uniform bounds \eqref{Westimate} on the weights, we find easily that
\begin{equation}\label{w11}
\begin{split}
&\frac{d}{dt}\int_\mathbb{R}J_0^-\mathcal{V}^-\,dx = \frac{d}{dt}\int_\mathbb{R}|w|\mathcal{V}^-\,dx \\
&\leq C\int_\mathbb{R} |w|(|R_1|+|S_1|)\mathcal{V}^-\,dx+C\int_\mathbb{R} |z||S_1|\mathcal{V}^+\,dx+ C\int_\mathbb{R}  \Big|v_1+\frac{R_1w-S_1z}{2c}\Big|\mathcal{V}^-\,dx\\
&\quad+G(t)\int_\mathbb{R} |w|\mathcal{V}^-\,dx-2c_0\int_\mathbb{R} |w|\mathbf{S}^2\mathcal{V}^-\,dx.
\end{split}
\end{equation}
Similar to the estimate of \eqref{w11}, we can also obtain the time derivative of  $\int_\mathbb{R}J_0^+\mathcal{V}^+\,dx$. Hence, it holds that
\begin{equation}\label{I0est}
\begin{split}
\frac{d}{dt}I_0 &=\frac{d}{dt}\int_\mathbb{R}
\Big(J_0^-\mathcal{V}^-+J_0^+\mathcal{V}^+\Big)\,dx\\
&\leq C\sum_{k=1,2}\int_\mathbb{R}\Big((1+|\mathbf{S}|)J_k^-\mathcal{V}^-
+(1+|\mathbf{R}|)J_k^+\mathcal{V}^+\Big)\,dx\\
&\quad+G(t)I_0-2c_0\int_\mathbb{R} \Big(\mathbf{S}^2J_0^-\mathcal{V}^-+\mathbf{R}^2J_0^+\mathcal{V}^+\Big)\,dx.
\end{split}
\end{equation}

\vspace{.2cm}
\paragraph{\bf Step 2.} Next, we estimate the time derivative of $I_1$. By \eqref{R-S-eqn} and the first equation of \eqref{balance}, we have
\begin{equation}\label{1+R}
(1+\mathbf{R}^2)_t-[c(1+\mathbf{R}^2)]_x=\frac{c'}{2c}
(\mathbf{R}^2S_1-R_1\mathbf{S}^2-R_1
+S_1).
\end{equation}
This together with \eqref{wz} yields
\begin{equation}\label{wt1}
\begin{split}
&\big[w(1+\mathbf{R}^2)\big]_t-\big[cw(1+\mathbf{R}^2)\big]_x\\
&=(w_t-cw_x)(1+\mathbf{R}^2)
+w[(1+\mathbf{R}^2)_t-\big(c(1+\mathbf{R}^2)\big)_x]\\
&=-c'\big(v_1+
\frac{R_1-S_1}{2c}w\big)(1+\mathbf{R}^2)+\frac{c'w}{2c}\big
(\mathbf{R}^2S_1-R_1\mathbf{S}^2-R_1
+S_1\big)\\
&=-c'\big(v_1+
\frac{R_1w-S_1z}{2c}\big)(1+\mathbf{R}^2)+\frac{c'w}{2c}\big
(2\mathbf{R}^2S_1-R_1\mathbf{S}^2-R_1
+2S_1\big)-\frac{c'z}{2c}S_1(1+\mathbf{R}^2).
\end{split}
\end{equation}
By summing up \eqref{Westimate} and \eqref{wt1}, we achieve
\begin{equation*}
\begin{split}
&\frac{d}{dt}\int_\mathbb{R}J_1^-\mathcal{V}^-\,dx = \frac{d}{dt}\int_\mathbb{R}|w|(1+\mathbf{R}^2)\mathcal{V}^-\,dx \\
&\leq C\int_\mathbb{R} |w|(|R_1\mathbf{S}^2|+|\mathbf{R}^2S_1|+|R_1|+|S_1|)\mathcal{V}^-\,dx\\
&\quad+C\int_\mathbb{R} |z|(|S_1|+|\mathbf{R}^2S_1|)\mathcal{V}^+\,dx+C\int_\mathbb{R}  \Big|v_1+\frac{R_1w-S_1z}{2c}\Big|(1+\mathbf{R}^2)\mathcal{V}^-\,dx\\
&\quad+G(t)\int_\mathbb{R} |w|(1+\mathbf{R}^2)\mathcal{V}^-\,dx-2c_0\int_\mathbb{R} |w|(1+\mathbf{R}^2)\mathbf{S}^2\mathcal{V}^-\,dx.
\end{split}
\end{equation*}
In a similar way, we can get the remaining term of $I_1$. We thus conclude from the Young inequality that
\begin{equation}\label{I1est}
\begin{split}
\frac{d}{dt}I_1
\leq& C\sum_{k=1,2}\int_\mathbb{R}\Big((1+|\mathbf{S}|)J_k^-\mathcal{V}^-
+(1+|\mathbf{R}|)J_k^+\mathcal{V}^+\Big)\,dx+ C\int_\mathbb{R}\Big(\mathbf{S}^2J_0^-\mathcal{V}^-+\mathbf{R}^2J_0^+\mathcal{V}^+\Big)\,dx\\
&
+G(t)I_1-c_0\int_\mathbb{R} \Big(\mathbf{S}^2J_1^-\mathcal{V}^-
+\mathbf{R}^2J_1^+\mathcal{V}^+\Big)\,dx.
\end{split}
\end{equation}

\vspace{.2cm}
\paragraph{\bf Step 3.}  For the time derivative of $I_2$, we first use
\eqref{vx} and \eqref{vt} to derive the equation for the first order perturbation $\mathbf{v}$ as
\begin{equation}\label{vequ}
\mathbf{v}_t-c\mathbf{v}_x=\mathbf{s}+
\frac{c'}{2c}(\mathbf{R}-\mathbf{S})v_1.
\end{equation}
On the other hand, it follows from \eqref{R-S-eqn} and \eqref{wz} that
\begin{equation}\label{Rwt}
\begin{split}
&\Big[\frac{R_iw-S_iz}{2c}\Big]_t
-c\Big[\frac{R_iw-S_iz}{2c}\Big]_x\\
&=\frac{w}{2c}(\partial_t R_i-c\partial_xR_i)
-\frac{z}{2c}(\partial_t S_i+c\partial_xS_i)+z\partial_xS_i
+\frac{R_i}{2c}(w_t-cw_x)\\
&\quad-\frac{S_i}{2c}(z_t+cz_x)+S_iz_x
+(R_iw-S_iz)\big[(\frac{1}{2c})_t-c(\frac{1}{2c})_x\big]\\
&=\frac{n_i}{8c^3}\Big[(c^2-\zeta_i)(\mathbf{R}^2+\mathbf{S}^2)
-2(3c^2-\zeta_i)\mathbf{R}\cdot\mathbf{S}\Big](w-z)+z\partial_xS_i+S_iz_x\\
&\quad-\frac{c'}{2c}\big(v_1+\frac{R_1w-S_1z}{2c}    \big)(R_i+S_i)
+\frac{c'}{4c^2}(R_1-S_1)(R_iw-S_iz),
\end{split}
\end{equation}
for $i=1,2,3$ and $ \zeta_1=\ga,$ $\zeta_2=\zeta_3=\al.$ Hence combining \eqref{1+R}, \eqref{vequ} and \eqref{Rwt}, we obtain
\begin{equation*}\label{Rwt1}
\begin{split}
&\Big[\big(v_i+\frac{R_iw-S_iz}{2c}\big)(1+\mathbf{R}^2)\Big]_t
-\Big[c\big(v_i+\frac{R_iw-S_iz}{2c}\big)(1+\mathbf{R}^2)\Big]_x\\
&=\Big[(\partial_tv_i-c\partial_xv_i)
+\big(\frac{R_iw-S_iz}{2c}\big)_t
-c\big(\frac{R_iw-S_iz}{2c}\big)_x\Big](1+\mathbf{R}^2)\\
&\quad+\big(v_i+\frac{R_iw-S_iz}{2c}\big)
\big[(1+\mathbf{R}^2)_t-\big(c(1+\mathbf{R}^2)\big)_x\big]\\
&=(1+\mathbf{R}^2)\Big[s^*_i+S_i(z_x+\frac{c'}{4c^2}(w-z)R_1)
-\frac{c'}{c}S_i(v_1+\frac{R_1w-S_1z}{2c}\big)\Big]
\\
&\quad+\frac{n_i}{8c^3}(c^2-\zeta_i)(w-z)(1+\mathbf{R}^2)\mathbf{S}^2
+\frac{c'}{2c}\big(v_i+\frac{R_iw-S_iz}{2c}\big)(\mathbf{R}^2S_1-
R_1\mathbf{S}^2-R_1+S_1).
\end{split}
\end{equation*}
This together with the uniform bounds on the weights \eqref{Westimate} implies that
\begin{equation*}\label{J2est}
\begin{split}
&\frac{d}{dt}\int_\mathbb{R}J_2^-\mathcal{V}^-\,dx = \frac{d}{dt}\sum_{i=1}^3\int_\mathbb{R}\Big|v_i+\frac{R_iw-S_iz}{2c}\Big|(1+\mathbf{R}^2)\mathcal{V}^-\,dx \\
&\leq C\sum_{i=1}^3\int_\mathbb{R}  \Big|v_i+\frac{R_iw-S_iz}{2c}\Big|\big[(|R_1|+|S_1|+|S_i|+|\mathbf{R}^2S_1|+|\mathbf{R}^2S_i|)\mathcal{V}^-
+|R_1\mathbf{S}^2|\mathcal{V}^+\big]\,dx\\
&\quad+C\int_\mathbb{R} \Big|s^*_i+S_i(z_x+\frac{c'}{4c^2}(w-z)R_1)\Big|(1+\mathbf{R}^2)\mathcal{V}^+\,dx
+C\int_\mathbb{R} |w|(1+\mathbf{R}^2)\mathbf{S}^2\mathcal{V}^-\,dx\\
&\quad+C\int_\mathbb{R} |z|(1+\mathbf{R}^2)\mathbf{S}^2\mathcal{V}^+\,dx+G(t)\sum_{i=1}^3\int_\mathbb{R} \Big|v_i+\frac{R_iw-S_iz}{2c}\Big|(1+\mathbf{R}^2)\mathcal{V}^-\,dx\\
&\quad
-2c_0\sum_{i=1}^3\int_\mathbb{R} \Big|v_i+\frac{R_iw-S_iz}{2c}\Big|(1+\mathbf{R}^2)\mathbf{S}^2\mathcal{V}^-\,dx.
\end{split}
\end{equation*}
By using the same argument to the time derivative of $\int_\mathbb{R}J_2^+\mathcal{V}^+\,dx$, we arrive at
\begin{equation}\label{I2est}
\begin{split}
\frac{d}{dt}I_2
\leq&  C\int_\mathbb{R}\Big((1+|\mathbf{S}|)J_2^-\mathcal{V}^-
+(1+|\mathbf{R}|)J_2^+\mathcal{V}^+\Big)\,dx\\
&+
C\sum_{k=1,4}\int_\mathbb{R}\Big((1+\mathbf{S}^2)J_k^-\mathcal{V}^-
+(1+\mathbf{R}^2)J_k^+\mathcal{V}^+\Big)\,dx\\
&+G(t)I_2
-2c_0\int_\mathbb{R} \Big(\mathbf{S}^2J_2^-\mathcal{V}^-+\mathbf{R}^2J_2^+\mathcal{V}^+\Big)\,dx.
\end{split}
\end{equation}

\vspace{.2cm}
\paragraph{\bf Step 4.}  Now we devote to the time derivative of $I_3$.
By \eqref{R-S-eqn} and \eqref{wz}, it holds that
\begin{equation}\label{Sest}
\begin{split}
&\Big[\frac{c'}{4c^2}(w-z)S_i\Big]_t
-\Big[c\frac{c'}{4c^2}(w-z)S_i\Big]_x\\
&=\frac{c'S_i}{4c^2}(w_t-cw_x)
-\frac{c'S_i}{4c^2}(z_t+cz_x)+\frac{c'S_i}{2c}z_x
+\frac{c'(w-z)}{4c^2}\big(\partial_tS_i+c\partial_xS_i\big)\\
&\quad
-\frac{c'(w-z)}{2c}\partial_x S_i+\frac{cc''-2(c')^2}{4c^3}(\partial_tn_1-c\partial_xn_1)(w-z)S_i-\frac{c'(w-z)}{4c^2}S_ic'\partial_x n_1\\
&=-\frac{(c')^2}{2c^2}\big(v_1+\frac{R_1w-S_1z}{2c}\big)S_i
+\frac{c'n_i}{16c^4}(w-z)\big[(c^2-\zeta_i)(\mathbf{R}^2+\mathbf{S}^2)
-2(3c^2-\zeta_i)\mathbf{R}\cdot\mathbf{S}\big]\\
&\quad
+\frac{c'S_i}{2c}z_x-\frac{c'(w-z)}{2c}\partial_x S_i+\frac{c''(w-z)}{4c^2}S_1 S_i-\frac{(c')^2}{8c^3}(w-z)(S_1 S_i+R_iS_1),
\end{split}
\end{equation}
for $i=1,2,3$ and $ \zeta_1=\ga, \zeta_2=\zeta_3=\al.$  Next, differentiating $\eqref{wz}_1$ with respect to $x$, it is clear that
\begin{equation}\label{wxt}
\begin{split}
w_{tx}&-\big(c w_x\big)_x=\frac{(c')^2-cc''}{2c^2}(R_1-S_1)(v_1+\frac{R_1-S_1}{2c}w)\\
&-\frac{c'}{2c}(r_1+w\partial_xR_1+R_1w_x)
+\frac{c'}{2c}(s_1+w\partial_xS_1+S_1w_x).
\end{split}\end{equation}
Putting \eqref{rseq}, \eqref{Sest} for $i=1$ and \eqref{wxt} together leads to\
\begin{equation}\label{wxtest}
\begin{split}
&\Big[w_x+\frac{c'}{4c^2}(w-z)S_1\Big]_t
-\Big[c\big(w_x+\frac{c'}{4c^2}(w-z)S_1\big)\Big]_x\\
&=w_{xt}-\big(cw_x\big)_x+\Big[\frac{c'}{4c^2}(w-z)S_1\Big]_t
-\Big[c\frac{c'}{4c^2}(w-z)S_1\Big]_x\\
&=-\frac{c'}{2c}(r_1+w\partial_xR_1+R_1w_x)+\frac{c'}{2c}(s_1+z\partial_xS_1+S_1z_x)
+\frac{c'}{2c}S_1\big(w_x+\frac{c'}{4c^2}(w-z)S_1\big)\\
&\quad+\frac{(c')^2-cc''}{2c^2}\big(v_1+\frac{R_1w-S_1z}{2c}\big)(R_1-S_1)
-\frac{(c')^2}{2c^2}\big(v_1+\frac{R_1w-S_1z}{2c}\big)S_1\\
&\quad+\frac{c'n_1}{16c^4}(w-z)\big[(c^2-\gamma)(\mathbf{R}^2+\mathbf{S}^2)
-2(3c^2-\gamma)\mathbf{R}\cdot\mathbf{S}\big]+\frac{2 c c''-3(c')^2}{8c^3}(w-z)R_1S_1\\
&=-\frac{c'}{2c}\big[r^*_1+R_1(w_x+\frac{c'}{4c^2}(w-z)S_1)\big]
+\frac{c'}{2c}\big[s^*_1+S_1(z_x+\frac{c'}{4c^2}(w-z)R_1)\big]\\
&\quad+\frac{c'}{2c}S_1\big(w_x+\frac{c'}{4c^2}(w-z)S_1\big)
+\frac{(c')^2-cc''}{2c^2}\big(v_1+\frac{R_1w-S_1z}{2c}\big)(R_1-S_1)\\
&\quad
-\frac{(c')^2}{2c^2}\big(v_1+\frac{R_1w-S_1z}{2c}\big)S_1+\frac{c'n_1}{8c^4}(w-z)\big[(c^2-\gamma)\mathbf{S}^2
-(3c^2-\gamma)\mathbf{R}\cdot\mathbf{S}\big]\\
&\quad+\frac{2 c c''-3(c')^2}{8c^3}(w-z)R_1S_1,
\end{split}
\end{equation}
from which one can deduce that
\begin{equation}\label{J3est}
\begin{split}
&\frac{d}{dt}\int_\mathbb{R}J_3^-\mathcal{V}^-\,dx = \frac{d}{dt}\int_\mathbb{R}\Big|w_x+\frac{c'}{4c^2}(w-z)S_1\Big|\mathcal{V}^-\,dx \\
&\leq C\int_\mathbb{R} |w|(|R_1S_1|+\mathbf{S}^2+|\mathbf{R}\cdot\mathbf{S}|)\mathcal{V}^-\,dx+C\int_\mathbb{R} |z|(|R_1S_1|+\mathbf{S}^2+|\mathbf{R}\cdot\mathbf{S}|)\mathcal{V}^+\,dx\\
&\quad+C\int_\mathbb{R}  \Big|v_1+\frac{R_1w-S_1z}{2c}\Big|(|R_1|+|S_1|)\mathcal{V}^-\,dx+C\int_\mathbb{R} \Big|w_x+\frac{c'}{4c^2}(w-z)S_1\Big||S_1|\mathcal{V}^-\,dx\\
&\quad+C\int_\mathbb{R} \Big[\Big|r^*_1+R_1(w_x+\frac{c'}{4c^2}(w-z)S_1)\Big|\mathcal{V}^-+ \Big|s^*_1+S_1(z_x+\frac{c'}{4c^2}(w-z)R_1)\Big|\mathcal{V}^+\Big]\,dx\\
&\quad
+G(t)\int_\mathbb{R} \Big|w_x+\frac{c'}{4c^2}(w-z)S_1\Big|\mathcal{V}^-\,dx
-2c_0\int_\mathbb{R} \Big|w_x+\frac{c'}{4c^2}(w-z)S_1\Big|S^2\mathcal{V}^-\,dx.
\end{split}
\end{equation}
Hence, arguing exactly as for the proof of \eqref{J3est} shows that
\begin{equation}\label{I3est}
\begin{split}
\frac{d}{dt}I_3
\leq&  C\sum_{k=2,3,4}\int_\mathbb{R}\Big((1+|\mathbf{S}|)J_k^-\mathcal{V}^-
+(1+|\mathbf{R}|)J_k^+\mathcal{V}^+\Big)\,dx\\
&+
C\int_\mathbb{R}\Big((1+\mathbf{S}^2)J_1^-\mathcal{V}^-
+(1+\mathbf{R}^2)J_1^+\mathcal{V}^+\Big)\,dx\\
&+G(t)I_3-2c_0\int_\mathbb{R} \Big(\mathbf{S}^2J_3^-\mathcal{V}^-+\mathbf{R}^2J_3^+\mathcal{V}^+\Big)\,dx.
\end{split}
\end{equation}

\vspace{.2cm}
\paragraph{\bf Step 5.}  Bounding the time derivative of $I_4$ is a little bit complicated. To this end, we first
differentiate $\eqref{R-S-eqn}_1$ with respect to $x$ and arrive at
\begin{equation}\label{Rxt}
\begin{split}
\partial_{xt}R_{i}&-\big(c\partial_xR_i\big)_x=
\frac{n_i}{2c^2}\big[(c^2-\zeta_i)(\mathbf{R}\cdot\mathbf{R}_x
+\mathbf{S}\cdot\mathbf{S}_x)-(3c^2-\zeta_i)
(\mathbf{R}_x\cdot\mathbf{S}+\mathbf{R}\cdot\mathbf{S}_x)\big]\\
&+\frac{cc''-(c')^2}{4c^3})(R_1-S_1)(R_i-S_i)R_1
+\frac{c'}{2c}\big[(\partial_xR_i-\partial_xS_i)R_1+(R_i-S_i)\partial_xR_1\big]\\
&+\frac{1}{8c^3}\big[(c^2-\zeta_i)(\mathbf{R}^2+\mathbf{S}^2)-
2(3c^2-\zeta_i)\mathbf{R}\cdot\mathbf{S}\big](R_i-S_i)\\
&+\frac{c'\lambda_in_i}{4c^4}(\mathbf{R}^2+\mathbf{S}^2-2\mathbf{R}\cdot\mathbf{S})(R_1-S_1)
,
\end{split}\end{equation}
for $i=1,2,3$ and $ \zeta_1=\ga, \zeta_2=\zeta_3=\al.$ Then with aid of \eqref{rt}, \eqref{wz} and \eqref{Rxt}, we deduce that
\begin{equation}\label{rtest1}
\begin{split}
&\big[r_i+w\partial_xR_i\big]_t-\big[c(r_i+w\partial_xR_i)\big]_x\\
&=\partial_tr_i-c\partial_xr_i-r_ic'\partial_x n_1
+\partial_xR_i(w_t-cw_x)
+w[(\partial_xR_{i})_t-(c\partial_xR_i)_x]\\
&=\Big[\frac{c'\lambda_in_i}{2c^3}(\mathbf{R}^2+\mathbf{S}^2-2\mathbf{R}\cdot\mathbf{S})
+\frac{cc''-(c')^2}{2c^2}(R_i-S_i)R_1\Big](v_1+\frac{R_1-S_1}{2c}w)
\\
&\quad+\frac{1}{4c^2}\big[(c^2-\zeta_i)(\mathbf{R}^2+\mathbf{S}^2)
-2(3c^2-\zeta_i)\mathbf{R}\cdot\mathbf{S}\big](v_i+\frac{R_i-S_i}{2c}w)\\
&\quad+\frac{c'}{2c}\big[(R_i-S_i)(r_1+w\partial_xR_1)-R_1(s_i+w\partial_xS_i)
+S_1(r_i+w\partial_xR_i)\big]\\
&\quad+\frac{n_i}{2c^2}(c^2-\zeta_i)\big[\mathbf{R}\cdot(\mathbf{r}+\mathbf{R}_xw)
+\mathbf{S}\cdot(\mathbf{s}+\mathbf{S}_xw)\big]\\
&\quad-\frac{n_i}{2c^2}(3c^2-\zeta_i)\big[\mathbf{R}\cdot(\mathbf{s}+\mathbf{S}_xw)
+\mathbf{S}\cdot(\mathbf{r}+\mathbf{R}_xw)\big].\\
\end{split}
\end{equation}
In addition, combining \eqref{R-S-eqn}, \eqref{balance} and \eqref{wz} directly gives
\begin{equation}\label{rtest2}
\begin{split}
&\Big[\frac{n_i}{8c^3}\Big((c^2-\zeta_i)\mathbf{S}^2
-2(3c^2-\zeta_i)\mathbf{R}\cdot\mathbf{S}\Big)(w-z)\Big]_t\\
&\quad-\Big[c\frac{n_i}{8c^3}\Big((c^2-\zeta_i)\mathbf{S}^2
-2(3c^2-\zeta_i)\mathbf{R}\cdot\mathbf{S}\Big)(w-z)\Big]_x\\
&=\frac{\partial_tn_i-c\partial_xn_i}{8c^3}\Big((c^2-\zeta_i)\mathbf{S}^2
-2(3c^2-\zeta_i)\mathbf{R}\cdot\mathbf{S}\Big)(w-z)\\
&\quad+\frac{n_i}{8c^3}\Big((c^2-\zeta_i)\mathbf{S}^2
-2(3c^2-\zeta_i)\mathbf{R}\cdot\mathbf{S}\Big)\big[(w_t-cw_x)
-(z_t+cz_x)+2cz_x\big]\\
&\quad+\frac{n_i(c^2-\zeta_i)}{8c^3}
\big[(\mathbf{S}^2)_t+\big(c\mathbf{S}^2\big)_x-2\big(c\mathbf{S}^2\big)_x\big](w-z)\\
&\quad-\frac{n_i(3c^2-\zeta_i)}{4c^3}
\big[(\mathbf{R}_t-c\mathbf{R}_x)\cdot\mathbf{S}
+\mathbf{R}\cdot(\mathbf{S}_t+c\mathbf{S}_x)-2c\mathbf{R}\cdot\mathbf{S}_x
-\mathbf{R}\cdot\mathbf{S}c'\partial_x n_1\big](w-z)\\
&\quad+\frac{n_i}{8}\Big(\mathbf{S}^2\big[\big(\frac{c^2-\zeta_i}{c^3}\big)_t
-c\big(\frac{c^2-\zeta_i}{c^3}\big)_x\big]
-2\mathbf{R}\cdot\mathbf{S}\big[
\big(\frac{3c^2-\zeta_i}{c^3}\big)_t
-c\big(\frac{3c^2-\zeta_i}{c^3}\big)_x\big]\Big)(w-z)\\
&=\frac{S_i}{8c^3}\Big((c^2-\zeta_i)\mathbf{S}^2
-2(3c^2-\zeta_i)\mathbf{R}\cdot\mathbf{S}\Big)(w-z)
-\frac{c^2-\zeta_i}{2c^2}n_i(w-z)\mathbf{S}\cdot\mathbf{S}_x\\
&\quad+\frac{3c^2-\zeta_i}{2c^2}n_i(w-z)\mathbf{R}\cdot\mathbf{S}_x
+\frac{n_iz_x}{4c^2}\Big((c^2-\zeta_i)\mathbf{S}^2
-2(3c^2-\zeta_i)\mathbf{R}\cdot\mathbf{S}\Big)\\
&\quad-\frac{c'n_i}{4c^3}\Big((c^2-\zeta_i)\mathbf{S}^2
-2(3c^2-\zeta_i)\mathbf{R}\cdot\mathbf{S}\Big)\big(v_1+\frac{R_1w-S_1z}{2c}\big)
\\
&\quad -\frac{c'\zeta_i n_i}{16c^4}(w-z)\big(R_1\mathbf{R}^2+2\mathbf{R}^2S_1-4R_1\mathbf{R}\cdot\mathbf{S}
+4S_1\mathbf{R}\cdot\mathbf{S}
+3R_1\mathbf{S}^2-2S_1\mathbf{S}^2\big)\\
&\quad +\frac{c'n_i}{16c^2}(w-z)\big(3R_1\mathbf{R}^2+8\mathbf{R}^2S_1-12R_1\mathbf{R}\cdot\mathbf{S}
-12S_1\mathbf{R}\cdot\mathbf{S}
+9R_1\mathbf{S}^2+4S_1\mathbf{S}^2\big).
\end{split}
\end{equation}
To continue, one can use \eqref{R-S-eqn} and \eqref{Sest} to get
\begin{equation}\label{rtest3}
\begin{split}
&\Big[\frac{c'}{4c^2}(w-z)R_1S_i\Big]_t
-\Big[c\frac{c'}{4c^2}(w-z)R_1S_i\Big]_x\\
&=(\partial_t R_1-c\partial_xR_1)\frac{c'}{4c^2}(w-z)S_i+R_1\big[\big(\frac{c'}{4c^2}(w-z)S_i\big)_t
-\big(c\frac{c'}{4c^2}(w-z)S_i\big)_x\big]\\
&=-\frac{(c')^2}{2c^2}\big(v_1+\frac{R_1w-S_1z}{2c}\big)R_1S_i
+\frac{c'R_1S_i}{2c}z_x-\frac{c'(w-z)}{2c}R_1\partial_x S_i\\
&\quad+\frac{c'}{16c^4}(w-z)R_1n_i\big[(c^2-\zeta_i)(\mathbf{R}^2+\mathbf{S}^2)
-2(3c^2-\zeta_i)\mathbf{R}\cdot\mathbf{S}\big]\\
&\quad+\frac{c'}{16c^4}(w-z)S_in_1\big[(c^2-\gamma)(\mathbf{R}^2+\mathbf{S}^2)
-2(3c^2-\gamma)\mathbf{R}\cdot\mathbf{S}\big]\\
&\quad
+\frac{c''(w-z)}{4c^2}R_1S_1 S_i-\frac{(c')^2}{8c^3}(w-z)(2R_1S_1 S_i+R_1R_iS_1-R_1^2S_i).
\end{split}
\end{equation}
In terms of \eqref{rseq} and \eqref{rtest1}--\eqref{rtest3}, the vertical displacement $\mathbf{r}^*$ satisfies the following equation
\begin{equation}\label{rtest5}
\begin{split}
&(r^*_i)_t-(cr^*_i)_x\\
&=\Big[\frac{c'\lambda_in_i}{2c^3}(\mathbf{R}^2+\mathbf{S}^2-2\mathbf{R}\cdot\mathbf{S})
+\frac{cc''-(c')^2}{2c^2}(R_i-S_i)R_1\Big](v_1+\frac{R_1w-S_1z}{2c})
\\
&\quad+\Big[\frac{(c')^2}{2c^2}R_1S_i-\frac{c'n_i}{4c^3}\big[(c^2-\zeta_i)\mathbf{S}^2
-2(3c^2-\zeta_i)\mathbf{R}\cdot\mathbf{S}\big]\Big](v_1+\frac{R_1w-S_1z}{2c})\\
&\quad+\frac{1}{4c^2}\big[(c^2-\zeta_i)(\mathbf{R}^2+\mathbf{S}^2)
-2(3c^2-\zeta_i)\mathbf{R}\cdot\mathbf{S}\big](v_i+\frac{R_iw-S_iz}{2c})\\
&\quad+\frac{n_i}{2c^2}(c^2-\zeta_i)\big[\mathbf{R}\cdot(\mathbf{r}+\mathbf{R}_xw)
+\mathbf{S}\cdot(\mathbf{s}+\mathbf{S}_xz)\big]\\
&\quad-\frac{n_i}{2c^2}(3c^2-\zeta_i)\big[\mathbf{R}\cdot(\mathbf{s}+\mathbf{S}_xz)
+\mathbf{S}\cdot(\mathbf{r}+\mathbf{R}_xw)\big]
\\
&\quad+\frac{c'}{2c}\big[(R_i-S_i)(r_1+w\partial_xR_1)-R_1(s_i+z\partial_xS_i+S_iz_x)
+S_1(r_i+w\partial_xR_i)\big]\\
&\quad-\frac{c'}{16c^4}(w-z)R_1n_i\big[(c^2-\zeta_i)(\mathbf{R}^2+\mathbf{S}^2)
-2(3c^2-\zeta_i)\mathbf{R}\cdot\mathbf{S}\big]\\
&\quad-\frac{c'}{16c^4}(w-z)S_in_1\big[(c^2-\gamma)(\mathbf{R}^2+\mathbf{S}^2)
-2(3c^2-\gamma)\mathbf{R}\cdot\mathbf{S}\big]\\
&\quad+\frac{(c')^2}{8c^3}(w-z)(3R_1R_iS_1-R_1^2S_i)
-\frac{c''}{4c^2}(w-z)R_1R_iS_1\\
&\quad-\frac{c^2-\zeta_i}{8c^3}(w-z)\mathbf{R}^2S_i
+\frac{n_iz_x}{4c^2}\Big((c^2-\zeta_i)\mathbf{S}^2
-2(3c^2-\zeta_i)\mathbf{R}\cdot\mathbf{S}\Big)\\
&\quad -\frac{c'\zeta_i n_i}{16c^4}(w-z)\big(R_1\mathbf{R}^2+6\mathbf{R}^2S_1-4R_1\mathbf{R}\cdot\mathbf{S}
-4S_1\mathbf{R}\cdot\mathbf{S}
+3R_1\mathbf{S}^2+2S_1\mathbf{S}^2\big)\\
&\quad +\frac{c'n_i}{16c^2}(w-z)\big(3R_1\mathbf{R}^2+8\mathbf{R}^2S_1-12R_1\mathbf{R}\cdot\mathbf{S}
-12S_1\mathbf{R}\cdot\mathbf{S}
+9R_1\mathbf{S}^2+4S_1\mathbf{S}^2\big).
\end{split}
\end{equation}

Moreover, it follows from \eqref{R-S-eqn} and \eqref{wxtest} that
\begin{equation}\label{Rwxtext}
\begin{split}
&\Big[R_i\big(w_x+\frac{c'}{4c^2}(w-z)S_1\big)\Big]_t
-\Big[cR_i\big(w_x+\frac{c'}{4c^2}(w-z)S_1\big)\Big]_x\\
&=(\partial_t R_i-c\partial_x R_i)\big(w_x+\frac{c'}{4c^2}(w-z)S_1\big)\\
&\quad+R_i\Big[\big(w_x+\frac{c'}{4c^2}(w-z)S_1\big)_t
-\big(c(w_x+\frac{c'}{4c^2}(w-z)S_1)\big)_x\Big]\\
&=\frac{(c')^2-cc''}{2c^2}\big(v_1+\frac{R_1w-S_1z}{2c}\big)(R_1-S_1)R_i
-\frac{(c')^2}{2c^2}\big(v_1+\frac{R_1w-S_1z}{2c}\big)R_iS_1\\
&\quad-\frac{c'}{2c}R_i(r_1+w\partial_xR_1)+\frac{c'}{2c}R_i(s_1
+z\partial_xS_1+S_1z_x)\\
&\quad
+\frac{c'}{2c}(R_iS_1-R_1S_i)\big(w_x+\frac{c'}{4c^2}(w-z)S_1\big)
+\frac{ c c''-(c')^2}{4c^3}(w-z)R_iR_1S_1\\
&\quad+\frac{n_i}{16c^4}\big(4c^2w_x+c'S_1(w-z)\big)\big[(c^2-\zeta_i)(\mathbf{R}^2+\mathbf{S}^2)
-2(3c^2-\zeta_i)\mathbf{R}\cdot\mathbf{S}\big]\\
&\quad+\frac{c'n_1}{16c^4}R_i(w-z)\big[(c^2-\gamma)(\mathbf{R}^2+\mathbf{S}^2)
-2(3c^2-\gamma)\mathbf{R}\cdot\mathbf{S}\big].
\end{split}
\end{equation}
Consequently, by \eqref{rtest5} and \eqref{Rwxtext}, it is easy to see that
\begin{equation*}
\begin{split}
&\Big[r_i^*+R_i\big(w_x+\frac{c'}{4c^2}(w-z)S_1\big)\Big]_t
-\Big[c\Big(r_i^*+R_i\big(w_x+\frac{c'}{4c^2}(w-z)S_1\big)\Big)\Big]_x\\
&=\frac{(c^2-\zeta_i)n_i}{4c^2}\big[2\mathbf{R}\cdot\mathbf{r}^*
+\mathbf{R}^2w_x+\frac{c'}{4c^2}(w-z)\mathbf{R}
^2S_1+2\mathbf{S}\cdot\mathbf{s}^*+\mathbf{S}^2z_x+\frac{c'}{4c^2}(w-z)R_1\mathbf{S}
^2\big]\\
&\quad-\frac{(3c^2-\zeta_i)n_i}{2c^2}\big[\mathbf{R}\cdot\big(\mathbf{s}^*
+\mathbf{S}(z_x+\frac{c'}{4c^2}(w-z)R_1)\big)
+\mathbf{S}\cdot(\mathbf{r}^*+\mathbf{R}(w_x+\frac{c'}{4c^2}(w-z)S_1)\big]
\\
&\quad-\frac{c'}{2c}R_1\big[s_i^*+S_i(z_x+\frac{c'}{4c^2}(w-z)R_1)\big]
+\frac{c'}{2c}R_i\big[s_1^*
+S_1(z_x+\frac{c'}{4c^2}(w-z)R_1)\big]\\
&\quad-\frac{c'}{2c}S_i\big[r^*_1+R_1(w_x+\frac{c'}{4c^2}(w-z)S_1)\big]
+\frac{c'}{2c}S_1\big[r^*_i
+R_i(w_x+\frac{c'}{4c^2}(w-z)S_1)\big]\\
&\quad+
\frac{(c^2-\zeta_i)n_i}{4c^2}\mathbf{S}^2(w_x+\frac{c'}{4c^2}(w-z)S_1\big)
+\frac{c'n_1}{16c^4}(c^2-\gamma)(w-z)\big(R_i\mathbf{S}^2-\mathbf{R}^2S_i\big)\\
&\quad+\Big[\frac{c'\lambda_in_i}{2c^3}(\mathbf{R}^2+\mathbf{S}^2-2\mathbf{R}\cdot\mathbf{S})
+\frac{cc''-2(c')^2}{2c^2}(R_iS_1-R_1S_i)\Big](v_1+\frac{R_1w-S_1z}{2c})
\\
&\quad-\frac{c'n_i}{4c^3}\big[(c^2-\zeta_i)\mathbf{S}^2
-2(3c^2-\zeta_i)\mathbf{R}\cdot\mathbf{S}\big](v_1+\frac{R_1w-S_1z}{2c})\\
&\quad+\frac{1}{4c^2}\big[(c^2-\zeta_i)(\mathbf{R}^2+\mathbf{S}^2)
-2(3c^2-\zeta_i)\mathbf{R}\cdot\mathbf{S}\big](v_i+\frac{R_iw-S_iz}{2c})\\
&\quad
-\frac{c^2-\zeta_i}{8c^3}(w-z)\mathbf{R}^2S_i-\frac{5c'\zeta_i n_i}{16c^4}(w-z)\mathbf{R}^2S_1+\frac{c'n_i}{16c^2}(w-z)\big(2R_1\mathbf{S}^2+3\mathbf{R}^2S_1\big),
\end{split}
\end{equation*}
which leads to
\begin{equation*}\label{ii}
\begin{split}
&\frac{d}{dt}\int_\mathbb{R}J_4^-\mathcal{V}^-\,dx = \frac{d}{dt}\sum_{i=1}^3\int_\mathbb{R}\big|r_i^*+R_i\big(w_x+\frac{c'}{4c^2}(w-z)S_1\big)\big|\mathcal{V}^-\,dx \\
&\leq
C\sum_{i=1}^3\int_\mathbb{R} \big|r_i^*+R_i\big(w_x+\frac{c'}{4c^2}(w-z)S_1\big)\big||\mathbf{S}|\mathcal{V}^-\,dx
\\
&\quad+C\sum_{i=1}^3\int_\mathbb{R} \big|s_i^*+S_i\big(z_x+\frac{c'}{4c^2}(w-z)R_1\big)\big||\mathbf{R}|\mathcal{V}^+\,dx
\\
&\quad +C\int_\mathbb{R} \Big|2\mathbf{R}\cdot\mathbf{r}^*
+\mathbf{R}^2w_x+\frac{c'}{4c^2}(w-z)\mathbf{R}
^2S_1\Big|\mathcal{V}^-\,dx+C\int_\mathbb{R} |w|(\mathbf{R}^2|\mathbf{S}|+|\mathbf{R}|\mathbf{S}^2)\mathcal{V}^-\,dx\\
&\quad
+C\int_\mathbb{R} \Big|2\mathbf{S}\cdot\mathbf{s}^*
+\mathbf{S}^2z_x+\frac{c'}{4c^2}(w-z)R_1\mathbf{S}
^2\Big|\mathcal{V}^+\,dx+C\int_\mathbb{R} |z|(\mathbf{R}^2|\mathbf{S}|+|\mathbf{R}|\mathbf{S}^2)\mathcal{V}^+\,dx
\\
&\quad
+C\sum_{i=1}^3\int_\mathbb{R}  \Big|v_i+\frac{R_iw-S_iz}{2c}\Big|(\mathbf{R}^2\mathcal{V}^-+\mathbf{S}^2\mathcal{V}^+
)\,dx+C\int_\mathbb{R} \Big|w_x+\frac{c'}{4c^2}(w-z)S_1\Big|\mathbf{S}^2\mathcal{V}^-\,dx\\
&\quad+G(t)
\sum_{i=1}^3\int_\mathbb{R}\big|r_i^*+R_i\big(w_x+\frac{c'}{4c^2}(w-z)S_1\big)\big|\mathcal{V}^-\,dx \\
&\quad-2c_0\sum_{i=1}^3\int_\mathbb{R} \big|r_i^*+R_i\big(w_x+\frac{c'}{4c^2}(w-z)S_1\big)\big|\mathbf{S}^2\mathcal{V}^-\,dx.
\end{split}
\end{equation*}

By performing a routine procedure, one can arrive at
\begin{equation}\label{I4est}
\begin{split}
\frac{d}{dt}I_4
\leq&  C\sum_{k=2,4,5}\int_\mathbb{R}\Big((1+|\mathbf{S}|)J_k^-\mathcal{V}^-
+(1+|\mathbf{R}|)J_k^+\mathcal{V}^+\Big)\,dx\\
&+
C\sum_{k=1,3}\int_\mathbb{R}\Big((1+\mathbf{S}^2)J_k^-\mathcal{V}^-
+(1+\mathbf{R}^2)J_k^+\mathcal{V}^+\Big)\,dx\\
&+G(t)I_4-2c_0\int_\mathbb{R} \Big(\mathbf{S}^2J_4^-\mathcal{V}^-+\mathbf{R}^2J_4^+\mathcal{V}^+\Big)\,dx.
\end{split}
\end{equation}

\vspace{.2cm}
\paragraph{\bf Step 6.}  In order to handle the last term of the norm \eqref{norm1}, we first observe that owing to  $|\mathbf{n}|=1$ and $c'(n_1)=\frac{(\gamma-\alpha)n_1}{c(n_1)}$, so that $\mathbf{R}\cdot\mathbf{n}=0$,  then it holds that
\begin{equation}\label{2Rest1}
\mathbf{R}\cdot\mathbf{r}^*=\mathbf{R}\cdot\mathbf{r}+\mathbf{R}\cdot\mathbf{R}_x w-\frac{c'}{8c^2}R_1\mathbf{S}^2(w-z).
\end{equation}
To estimate the time derivative of $\mathbf{R}\cdot\mathbf{r}^*$, using \eqref{balance}, we first compute
\begin{equation}\label{2Rest2}
\begin{split}
\big[\mathbf{R}\cdot\mathbf{r}\big]_t-\big[c\mathbf{R}\cdot\mathbf{r}\big]_x
=&c'v_1\mathbf{R}\mathbf{R}_x+\frac{cc''-(c')^2}{4c^2}(R_1\mathbf{R}^2-R_1\mathbf{S}^2)v_1
\\
&+\frac{c'}{4c}\big[(\mathbf{R}^2-\mathbf{S}^2)r_1
+2S_1\mathbf{R}\cdot\mathbf{r}-2R_1\mathbf{S}\cdot\mathbf{s}\big].
\end{split}
\end{equation}
Next, by \eqref{R-S-eqn}, \eqref{balance} and \eqref{wz}, we have
\begin{equation}\label{2Rest3}
\begin{split}
&\Big[\frac{c'}{8c^2}(w-z)R_1\mathbf{S}^2\Big]_t
-\Big[c\frac{c'}{8c^2}(w-z)R_1\mathbf{S}^2\Big]_x\\
&=\frac{c'}{8c^2}(w_t-cw_x)R_1\mathbf{S}^2-\frac{c'}{8c^2}(z_t+cz_x)R_1\mathbf{S}^2
+\frac{c'}{4c}z_xR_1\mathbf{S}^2+\frac{c'\mathbf{S}^2}{8c^2}(w-z)(\partial_t R_1-c\partial_xR_1)\\
&\quad+\frac{c'R_1}{8c^2}(w-z)\big(\partial_t \mathbf{S}^2+(c\mathbf{S}^2)_x\big)-\frac{c'R_1}{4c^2}(w-z)(c\mathbf{S}^2)_x
+R_1\mathbf{S}^2(w-z)\big[\big(\frac{c'}{8c^2}\big)_t
-c\big(\frac{c'}{8c^2}\big)_x\big]\\
&=-\frac{(c')^2}{4c^2}\big(v_1+\frac{R_1w-S_1z}{2c}\big)R_1\mathbf{S}^2
+\frac{c'R_1\mathbf{S}^2}{4c}z_x-\frac{c'(w-z)}{2c}R_1\mathbf{S} \cdot\mathbf{S}_x\\
&\quad+\frac{c'n_1\mathbf{S}^2}{32c^4}(w-z)\big[(c^2-\gamma)(\mathbf{R}^2+\mathbf{S}^2)
-2(3c^2-\gamma)\mathbf{R}\cdot\mathbf{S}\big]
\\
&\quad+\frac{cc''-(c')^2}{8c^3}(w-z)R_1S_1\mathbf{S}^2
+\frac{(c')^2}{16c^3}(w-z)R_1(R_1\mathbf{S}^2-S_1\mathbf{R}^2).
\end{split}
\end{equation}
Thanks to \eqref{R-S-eqn}, \eqref{wz}, \eqref{Rxt} and \eqref{2Rest1}--\eqref{2Rest3}, one can get
\begin{equation}\label{2Rest4}
\begin{split}
&\big[\mathbf{R}\cdot\mathbf{r}^*\big]_t-\big[c\mathbf{R}\cdot\mathbf{r}^*\big]_x\\
&=\big[\mathbf{R}\cdot\mathbf{r}\big]_t-\big[c\mathbf{R}\cdot\mathbf{r}\big]_x
+\mathbf{R}\cdot\mathbf{R}_x(w_t-cw_x)+w(\mathbf{R}_t-c\mathbf{R}_x)\cdot\mathbf{R}_x\\
&\quad+\mathbf{R}\cdot\big(\mathbf{R}_{xt}-(c\mathbf{R}_x)_x\big)w
-\Big[\frac{c'}{8c^2}(w-z)R_1\mathbf{S}^2\Big]_t
+\Big[c\frac{c'}{8c^2}(w-z)R_1\mathbf{S}^2\Big]_x\\
&=\frac{cc''-(c')^2}{4c^2}\big(v_1+\frac{R_1w-S_1z}{2c}\big)(\mathbf{R}^2-
\mathbf{S}^2)R_1+\frac{(c')^2}{4c^2}\big(v_1+\frac{R_1w-S_1z}{2c}\big)R_1\mathbf{S}^2
\\
&\quad-\frac{c'}{2c}R_1\mathbf{S} \cdot(\mathbf{s}+\mathbf{S}_xz)+\frac{c'}{2c}S_1\mathbf{R} \cdot(\mathbf{r}+\mathbf{R}_xw)+\frac{c'}{4c}
(\mathbf{R}^2-\mathbf{S}^2)(r_1+w\partial_x R_1)\\
&\quad-\frac{c'R_1\mathbf{S}^2}{4c}z_x-\frac{c'n_1\mathbf{S}^2}{32c^4}(w-z)\big[(c^2-\gamma)(\mathbf{R}^2+\mathbf{S}^2)
-2(3c^2-\gamma)\mathbf{R}\cdot\mathbf{S}\big]
\\
&\quad-\frac{cc''-(c')^2}{8c^3}(w-z)R_1S_1\mathbf{R}^2-\frac{(c')^2}{16c^3}(w-z)R_1(R_1\mathbf{S}^2-S_1\mathbf{R}^2).
\end{split}
\end{equation}
Hence, together \eqref{balance}, \eqref{wxtest} and \eqref{2Rest4}, we can conclude that
\begin{equation*}\label{wRest5}
\begin{split}
&\Big[2\mathbf{R}\cdot\mathbf{r}^*+\mathbf{R}^2w_x
+\frac{c'}{4c^2}(w-z)\mathbf{R}^2S_1\Big]_t
-\Big[c\big(2\mathbf{R}\cdot\mathbf{r}^*+\mathbf{R}^2w_x
+\frac{c'}{4c^2}(w-z)\mathbf{R}^2S_1\big)\Big]_x\\
&=2\big([\mathbf{R}\cdot\mathbf{r}^*]_t-[c\mathbf{R}\cdot\mathbf{r}^*]_x\big)
+\mathbf{R}^2\Big(\big[w_x+\frac{c'}{4c^2}(w-z)S_1\big]_t
-\big[c\big(w_x+\frac{c'}{4c^2}(w-z)S_1\big)\big]_x\Big)\\
&\quad+\big(w_x+\frac{c'}{4c^2}(w-z)S_1\big)[(\mathbf{R}^2)_x-c(\mathbf{R}^2)_x]\\
&=\frac{c'}{2c}S_1\big[2\mathbf{R} \cdot\mathbf{r}^*+\mathbf{R}^2w_x+\frac{c'}{4c^2}(w-z)\mathbf{R}^2S_1\big]
-\frac{c'}{2c}\mathbf{S}^2\big(r_1^*+ R_1w_x+\frac{c'}{4c^2}(w-z)R_1S_1\big)\\
&\quad-\frac{c'}{2c}R_1\big[2\mathbf{S}\cdot\mathbf{s}^*+\mathbf{S}^2z_x
+\frac{c'}{4c^2}(w-z)\mathbf{S}^2R_1\big]
+\frac{c'}{2c}\mathbf{R}^2\big(s_1^*+S_1z_x+\frac{c'}{4c^2}(w-z)R_1S_1\big)\\
&\quad+\frac{cc''-2(c')^2}{2c^2}\big(v_1+\frac{R_1w-S_1z}{2c}\big)
(\mathbf{R}^2S_1-R_1\mathbf{S}^2).
\end{split}
\end{equation*}
This in turn yields
\begin{equation*}
\begin{split}
&\frac{d}{dt}\int_\mathbb{R}J_5^-\mathcal{V}^-\,dx = \frac{d}{dt}\int_\mathbb{R}\Big|2\mathbf{R}\cdot\mathbf{r}^*+\mathbf{R}^2w_x
+\frac{c'}{4c^2}(w-z)\mathbf{R}^2S_1\Big|\mathcal{V}^-\,dx \\
&\leq C\int_\mathbb{R}  \Big|v_1+\frac{R_1w-S_1z}{2c}\Big|\mathbf{R}^2|S_1|\mathcal{V}^-\,dx
+C\int_\mathbb{R}  \Big|v_1+\frac{R_1w-S_1z}{2c}\Big||R_1|\mathbf{S}^2\mathcal{V}^+\,dx  \\
&\quad+ C\int_\mathbb{R} \big|r_1^*+ R_1w_x+\frac{c'}{4c^2}(w-z)R_1S_1\big|\mathbf{S}^2\mathcal{V}^-\,dx\\
&\quad+C\int_\mathbb{R} \big|s_1^*+S_1z_x+\frac{c'}{4c^2}(w-z)R_1S_1\big|\mathbf{R}^2\mathcal{V}^+\,dx\\
&\quad+C\int_\mathbb{R}\Big|2\mathbf{R}\cdot\mathbf{r}^*+\mathbf{R}^2w_x
+\frac{c'}{4c^2}(w-z)\mathbf{R}^2S_1\Big||S_1|\mathcal{V}^-\,dx
\\
&\quad+C\int_\mathbb{R} \Big|2\mathbf{S}\cdot\mathbf{s}^*+\mathbf{S}^2z_x
+\frac{c'}{4c^2}(w-z)\mathbf{S}^2R_1\Big||R_1|\mathcal{V}^+\,dx
\\
&\quad+G(t)\int_\mathbb{R} \Big|2\mathbf{R}\cdot\mathbf{r}^*+\mathbf{R}^2w_x
+\frac{c'}{4c^2}(w-z)\mathbf{R}^2S_1\Big|\mathcal{V}^-\,dx\\
&\quad
-2c_0\int_\mathbb{R} \Big|2\mathbf{R}\cdot\mathbf{r}^*+\mathbf{R}^2w_x
+\frac{c'}{4c^2}(w-z)\mathbf{R}^2S_1\Big|\mathbf{S}^2\mathcal{V}^-\,dx.
\end{split}
\end{equation*}
Therefore, we get the following estimate
\begin{equation}\label{I5est}
\begin{split}
\frac{d}{dt}I_5
\leq&  C\sum_{k=2,5}\int_\mathbb{R}\Big((1+|S_1|)J_k^-\mathcal{V}^-
+(1+|R_1|)J_k^+\mathcal{V}^+\Big)\,dx\\
&+C\int_\mathbb{R}\Big((1+\mathbf{S}^2)J_4^-\mathcal{V}^-
+(1+\mathbf{R}^2)J_4^+\mathcal{V}^+\Big)\,dx\\
&+G(t)I_5-2c_0\int_\mathbb{R} \Big(\mathbf{S}^2J_5^-\mathcal{V}^-+\mathbf{R}^2J_5^+\mathcal{V}^+\Big)\,dx.
\end{split}
\end{equation}

Summing up \eqref{I0est}, \eqref{I1est}, \eqref{I2est}, \eqref{I3est}, \eqref{I4est}, \eqref{I5est} and using \eqref{4.9}, we obtain the desired estimate \eqref{sum}.
\end{proof}

\end{appendices}

\section*{Acknowledgments}
The first author is partially supported by the National Natural Science Foundation of China (No. 11801295), and the Shandong Provincial Natural Science Foundation, China (No. ZR2018BA008). The second author is partially
supported by NSF with grant DMS-2008504. The third author is partially supported by NSF with grant DMS-206218.

\end{document}